\documentclass[twoside,12pt,a4paper]{amsart} 
\usepackage[utf8]{inputenc}
\usepackage[english]{babel}

\usepackage{amsmath}
\usepackage{amsfonts}
\usepackage{amssymb}
\usepackage{color}

\usepackage{amsthm}

\newtheorem{theorem}{Theorem}[section]

\newtheorem{lemma}[theorem]{Lemma}
\newtheorem{corollary}[theorem]{Corollary}
\newtheorem*{open}{Open Question}

\usepackage{graphics}
\usepackage{epsfig}

\usepackage{url}
\usepackage{hyperref}

\usepackage[a4paper,top=2.5cm,bottom=2.8cm,
left=3.2cm,right=3.2cm]{geometry}
\markleft{\hfill \textsc{Relationships Between Circles Inscribed in Triangles and Curvilinear Triangles} \hfill}
\begin{document}
Sangaku Journal of Mathematics (SJM) \copyright SJM \\
ISSN 2534-9562 \\
Volume 4 (2020), pp.9-27  \\
Received 19 November 2019. Published on-line 16 January 2020 \\ 
web: \url{http://www.sangaku-journal.eu/} \\
\copyright The Author(s) This article is published 
with open access\footnote{This article is distributed under the terms of the Creative Commons Attribution License which permits any use, distribution, and reproduction in any medium, provided the original author(s) and the source are credited.}. \\
\bigskip
\bigskip

\def\oldarc #1#2{\hbox{\raise 3.5pt\hbox{$\frown \atop \textstyle #1#2$}}}
\newcommand{\tarc}{\mbox{\large$\frown$}}
\newcommand{\marc}[1]{\stackrel{\tarc}{#1}}
\newcommand{\arc}[2]{\marc{#1#2}}
\newcommand{\Wedge}{skewed sector}
\newcommand{\Wedges}{skewed sectors}
\edef\void#1{}

\begin{center}
{\Large \textbf{Relationships Between Circles Inscribed in Triangles and Related Curvilinear Triangles}} \\
\medskip
\bigskip
\textsc{Stanley Rabinowitz} \\
545 Elm St Unit 1,  Milford, New Hampshire 03055, USA \\
e-mail: \href{mailto:stan.rabinowitz@comcast.net}{stan.rabinowitz@comcast.net} \\
web: \url{http://www.StanleyRabinowitz.com/} \\
\end{center}
\bigskip

\textbf{Abstract.} If $P$ is a point inside $\triangle ABC$, then the cevians
through $P$ extended to the circumcircle of $\triangle ABC$
create a figure containing a number of curvilinear triangles.
Each curvilinear triangle is bounded by an arc of the circumcircle and two line segments
lying along the sides or cevians of the original triangle.
We give theorems about the relationships between the radii of
circles inscribed in various sets of these curvilinear triangles.



\newcommand{\degrees}{^\circ}

\bigskip
\bigskip
\section{Introduction}

A \textit{curvilinear triangle} is a geometric figure bounded by three curves. The curves are typically line segments and arcs of circles, in which case there is a unique circle tangent to each of the three boundary curves. This circle is called the \textit{incircle} of the curvilinear triangle.

Wasan geometers loved to find relationships between the radii of circles inscribed in curvilinear triangles.
An example is shown in Figure~\ref{fig:Yamamoto} which comes from an 1841 book of Mathematical Formulae
written by Yamamoto \cite{Yamamoto}. It is also given as problem 5.3.9 in \cite{Fukagawa-Rigby}. In the figure,
$AH\perp BC$ and $BA\perp AC$. There are three
curvilinear triangles of interest in the figure.
The first curvilinear triangle is bounded by $BH$, $HA$, and arc $\arc{A}{B}$.
The second curvilinear triangle is bounded by $AH$, $HC$, and arc $\arc{C}{A}$.
The third curvilinear triangle is bounded by $CA$, $AB$, and arc $\arc{B}{C}$.
The radii of the circles inscribed in these curvilinear triangles are $r_1$, $r_2$, and $r_3$, respectively.
Then the nice relationship that was found is $r_1+r_2=r_3$.

\begin{figure}[h!t]
\centering
\includegraphics[width=0.5\linewidth]{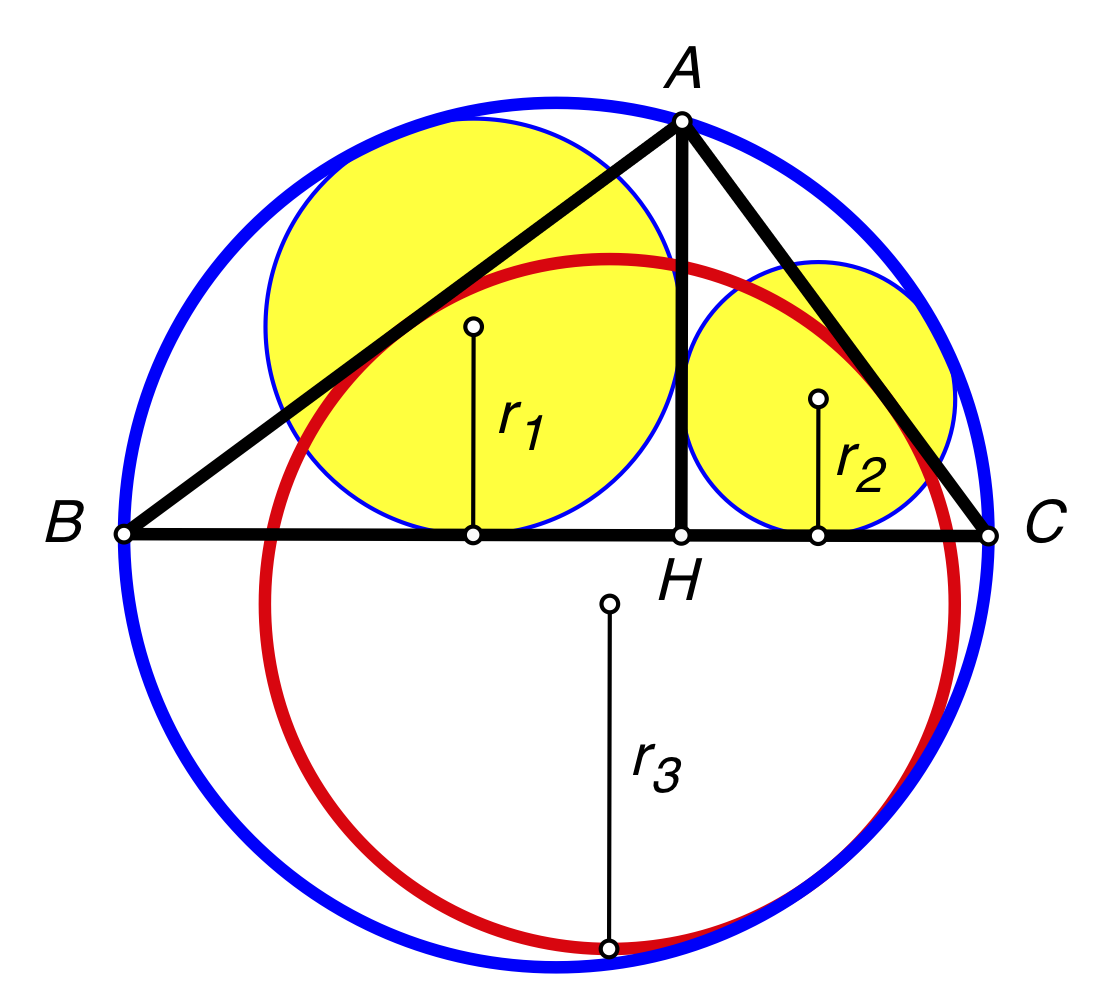}
\caption{$r_1+r_2=r_3$}
\label{fig:Yamamoto}
\end{figure}

In this paper, we will find some other nice relationships between the inradii of curvilinear triangles.

If a curvilinear triangle is convex and bounded by two straight line segments
and one circular arc, then we will call the resulting figure a \textit{\Wedge}
(see Figure~\ref{fig:wedge}).

\begin{figure}[h!t]
\centering
\includegraphics[width=0.35\linewidth]{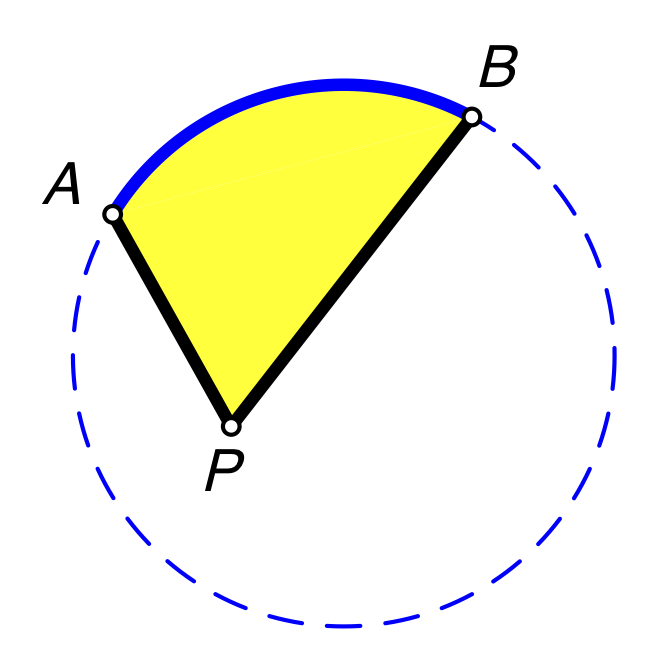}
\caption{a \Wedge}
\label{fig:wedge}
\end{figure}

\goodbreak
\textbf{Anatomy of a \Wedge.}

\begin{itemize}
\item The two straight line segments are called the \textit{sides} of the \Wedge.
\item The point of intersection of the two sides is called the \textit{vertex} of the \Wedge.
\item The angle between the two sides is called the \textit{vertex angle}.
\item The circular arc is referred to as the \textit{arc of the \Wedge}.
\item The circular measure of the arc of a {\Wedge} is called the \textit{arc angle}.
\item The circle to which the arc belongs will be called the \textit{circle associated with the \Wedge}.
\item The triangle formed by the vertex of a {\Wedge} and the endpoints of its arc will be
referred to as the \textit{triangle associated with the \Wedge}. This would be $\triangle APB$
in Figure~\ref{fig:wedge}.
\item When naming a \Wedge, the vertex will always be the middle letter.
Thus, the {\Wedge} in Figure~\ref{fig:wedge} is named {\Wedge} $APB$.
\item When the vertex of a {\Wedge} lies inside the associated circle, if the sides of the {\Wedge} are extended back through the vertex, they will intercept
an arc of the associated circle. This arc is called the \textit{opposite arc} of the \Wedge.
It is shown in red in Figure~\ref{fig:wedgeAnatomy}.
\item The vertex of a {\Wedge} and the opposite arc form another {\Wedge} called the \textit{opposite \Wedge}.
This is {\Wedge} $A'PB'$ in Figure~\ref{fig:wedgeAnatomy}.
\end{itemize}

\begin{figure}[h!t]
\centering
\includegraphics[width=0.4\linewidth]{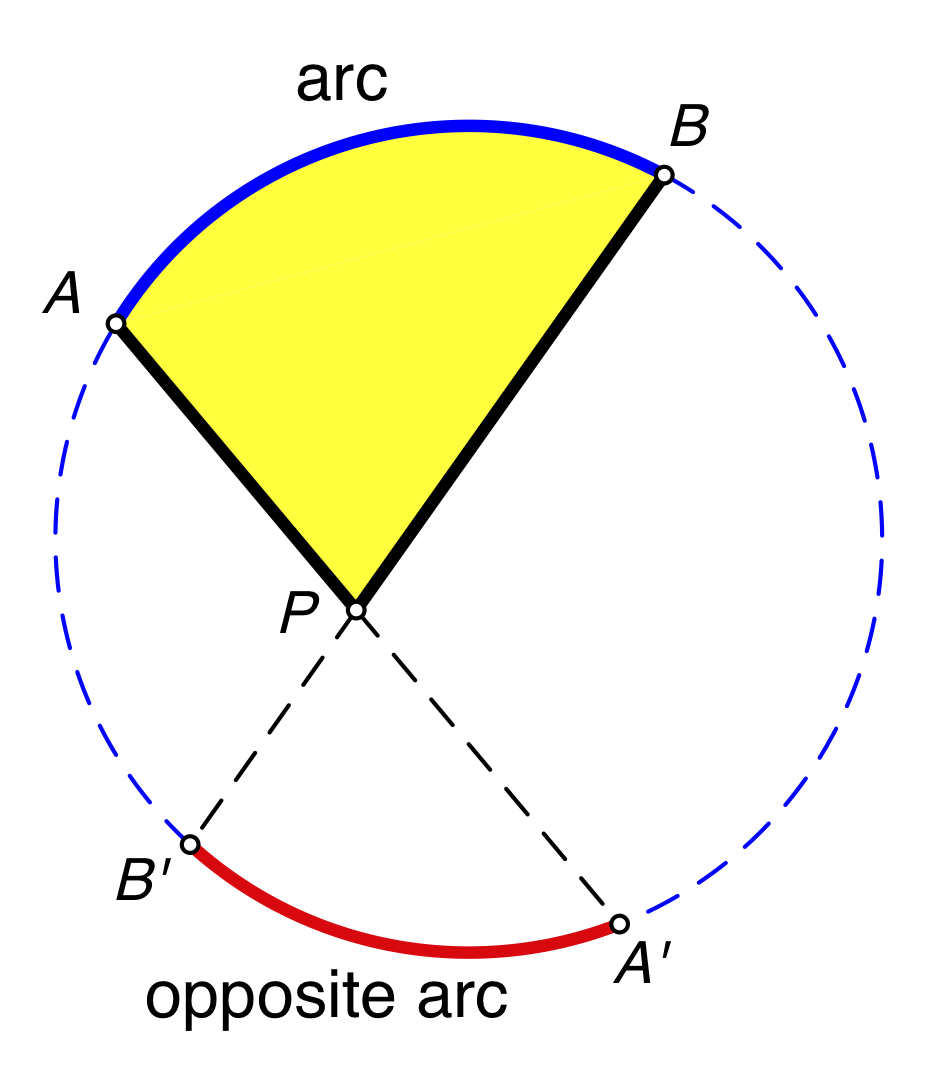}
\caption{opposite arc}
\label{fig:wedgeAnatomy}
\end{figure}

A \textit{segment} of a circle is the figure bounded by an arc of a circle and the chord joining the endpoints of that arc. The height (or sagitta) of the segment is the distance from the midpoint of the chord to the midpoint of the arc.

If $P$ is a point inside $\triangle ABC$, then the cevians
through $P$ extended to the circumcircle of $\triangle ABC$
create a figure containing a number of \Wedges.
We will find relationships between the radii of the circles inscribed in
some of these \Wedges.

\bigskip
\textbf{Notation.}
\begin{itemize}
\item If $X$ and $Y$ are points, then we use the notation $XY$ to denote either the line segment joining $X$ and $Y$ or the length of that line segment, depending on the context.
\item A \textit{cevian} of a triangle is a line segment from a vertex to the opposite side.
\item We use the notation $\angle XYZ$ to denote either the angle between $XY$ and $YZ$ or the measure of that angle, depending on the context.
\item The notation [XYZ] denotes the area of $\triangle XYZ$.
\item The notation $O(r)$ refers to the circle centered at point $O$ with radius $r$.
The circle may sometimes also be referred to as circle $O$.
\item If $\arc{X}{Y}$ is an arc of a circle, then $m(\arc{X}{Y})$ denotes the circular measure
of that arc. The arc extends counterclockwise along the circle from $X$ to $Y$.
\item Typically, we use $r$ for the inradius of a triangle and $w$ for the inradius of a \Wedge.
\end{itemize}

\section{Inradius Formula}

Formulas for the radius of the circle inscribed in a {\Wedge} and in a triangle are known.
Since these are not well-known, we review them here.

\newcommand{\C}{$\mathcal{C}$}

\begin{theorem}[Inradius of \Wedge]
\label{thm:wedgeFormula}
Let $APB$ be a {\Wedge} and let \C\ be the circle associated with arc $\arc AB$.
Suppose $P$ lies inside \C. Let $R$ be the radius of \C, let $w$ be the
radius of the circle inscribed in the \Wedge, and let $r$ be the radius of the circle inscribed in $\triangle APB$. Extend $BP$ to meet the circle \C\ at $C$ and draw $AC$.
Let $\alpha$, $\beta$, $\gamma$, $\delta$, and $\epsilon$ be the measures of five angles
associated with the {\Wedge} as shown in Figure~\ref{fig:wedgeFormula}.
Then
$$
\begin{aligned}
w&=\frac{\displaystyle 4R\sin{\beta\over2}\sin{\gamma\over2}\cos{\delta\over2}\sin{\epsilon\over2}}
{\displaystyle\left(\cos{\alpha\over2}\right)^2},\\[5pt]
r&=\frac{\displaystyle 4R\sin{\beta\over2}\sin{\gamma\over2}\sin{\epsilon\over2}\cos{\epsilon\over2}}
{\displaystyle\cos{\alpha\over2}}.
\end{aligned}
$$
\end{theorem}

\begin{figure}[h!t]
\centering
\includegraphics[width=0.4\linewidth]{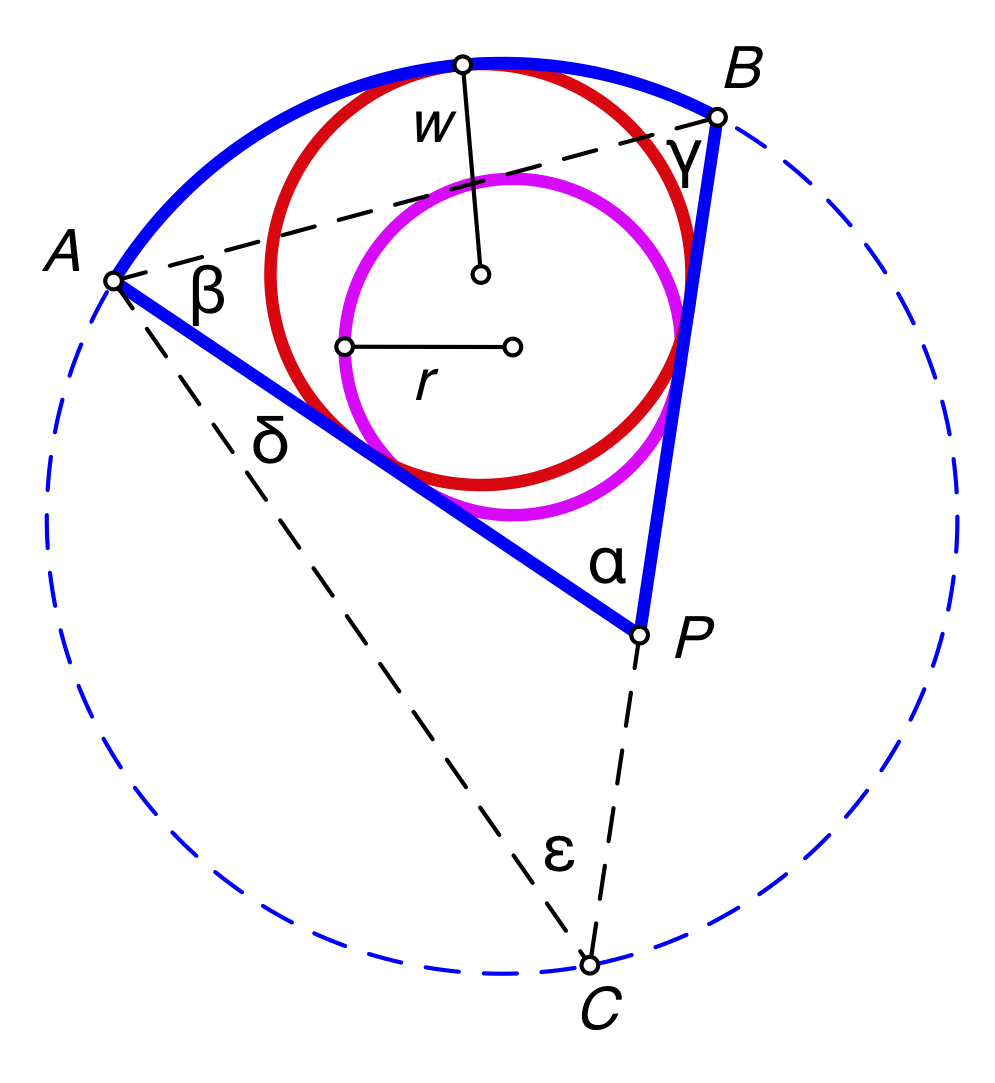}
\caption{angles associated with a \Wedge}
\label{fig:wedgeFormula}
\end{figure}

\begin{proof}
See \cite[pp.~96--97]{Fukagawa-Rigby} or \cite[p.~26]{Unger30}.
\end{proof}

An immediate consequence of this theorem is the following result.

\begin{theorem}
\label{thm:wrFormula}
Let $APB$ be a {\Wedge} and suppose $P$ lies inside the associated circle.
Let $w$ be the radius of the circle inscribed in the \Wedge, and
let $r$ be the radius of the circle inscribed in $\triangle APB$.
Let $\alpha$ be the vertex angle of the \Wedge,
let $\theta_1$ be the arc angle of the \Wedge, and let $\theta_2$ be the arc angle of the opposite \Wedge.
See Figure~\ref{fig:wrFormula}.
Then
$$\frac{w}{r}=\frac{\cos(\theta_2/4)}{\cos(\alpha/2)\cos(\theta_1/4)}.$$
\end{theorem}

\begin{figure}[h!t]
\centering
\includegraphics[width=0.45\linewidth]{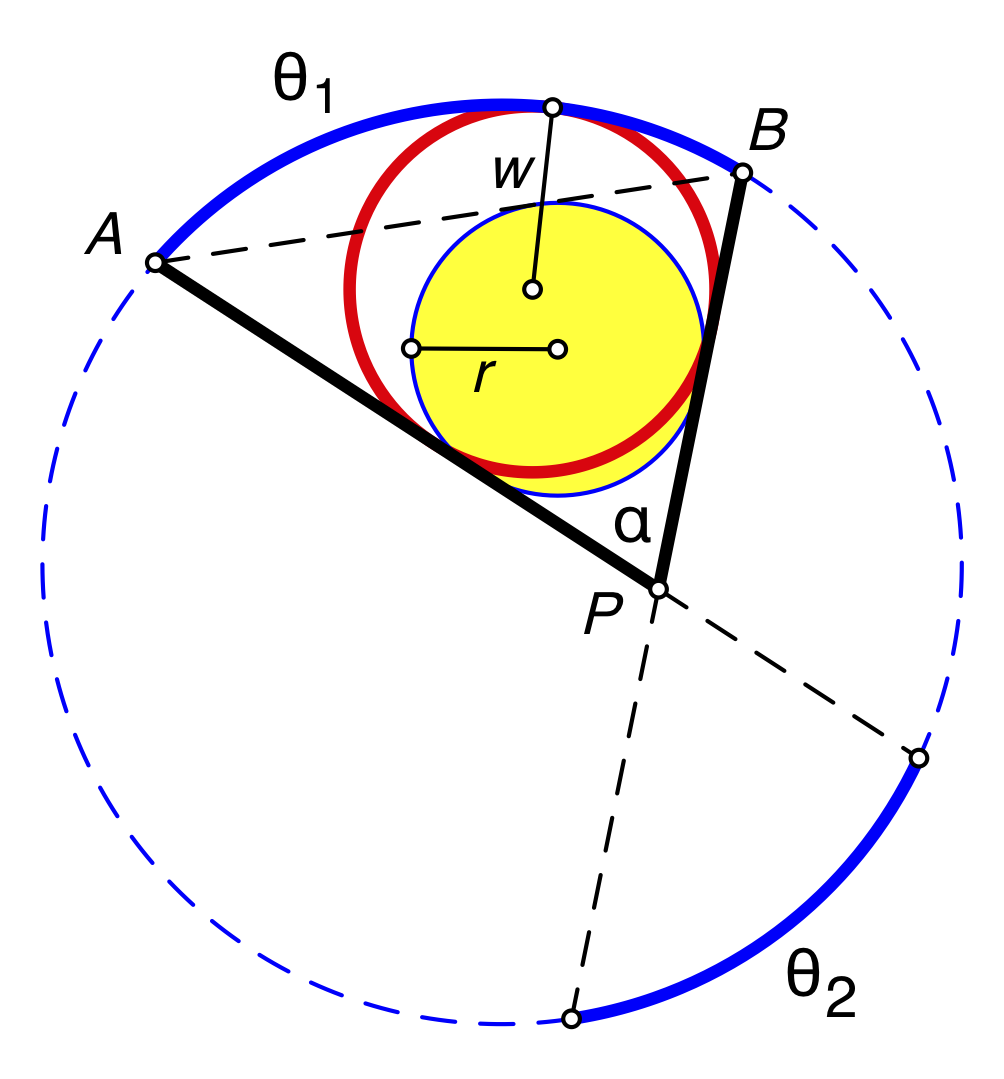}
\caption{}
\label{fig:wrFormula}
\end{figure}

We can also express $w/r$ without using $\theta_2$ as follows.

\begin{theorem}
\label{thm:wrFormula2}
Using the same notation as in Theorem \ref{thm:wrFormula},
$$\frac{w}{r}=1+\tan(\alpha/2)\tan(\theta_1/4).$$
\end{theorem}

\begin{proof}
See \cite[pp.~26--27]{Unger30}.
\end{proof}

\section{Relationship Between Incircles of Skewed Sectors and Incircles of Triangles}

To prove a relationship between {\Wedge} inradii, Theorems~\ref{thm:wedgeFormula}, \ref{thm:wrFormula}, or \ref{thm:wrFormula2} could be used
to find the length of each radius. This is a brute force technique and better methods are available.
One strategy for finding relationships between the radii of circles inscribed in {\Wedges} is to relate these
circles to circles inscribed in triangles, for which results are already known.

\void{
The cevians through a point $P$ inside a triangle $ABC$ divide that triangle into smaller triangles. The six small triangles with vertex at $P$ will be named $T_1$ through $T_6$ as shown in Figure~\ref{fig:smallTriangles}.
\begin{figure}[h!t]
\centering
\includegraphics[width=0.5\linewidth]{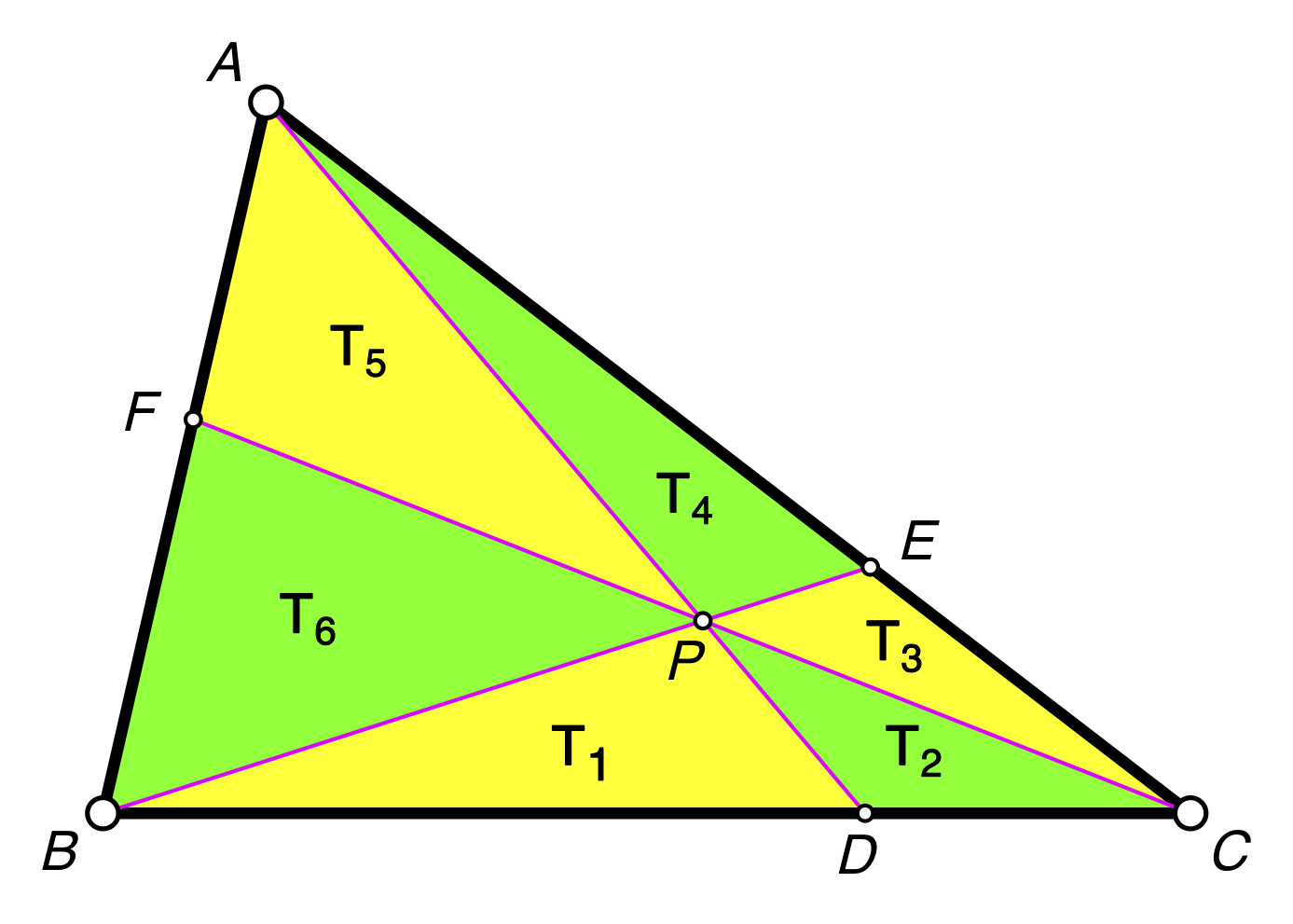}
\caption{naming of small triangles}
\label{fig:smallTriangles}
\end{figure}
}

Here are some theorems that relate circles in {\Wedges} to circles in triangles.

The following theorem appeared on a tablet in 1781.
See \cite[problem~4.0.3]{Fukagawa-Rigby}, \cite[problem~2.2.8]{Fukagawa-Pedoe}, \cite{Hirayama}, and \cite{Sanpo}.

\begin{theorem}[Ajima's Theorem]
\label{thm:AjimasTheorem}
Let $AB$ be a chord of a circle and let $C$ be a point inside the circle, not on the chord.
See Figure~\ref{fig:AjimasTheorem}.
Let $W(w)$ be the incircle of {\Wedge} $ACB$ (the red circle) and let $O(r)$ be the
incircle of $\triangle ACB$ (the yellow circle).
Then
$$w=r+\frac{2d(s-a)(s-b)}{cs},$$
where $d$ is the height of the segment formed by $AB$, $a=BC$, $b=AC$, $c=AB$,
and $s$ is the semiperimeter of $\triangle ABC$.
\end{theorem}

\begin{figure}[h!t]
\centering
\includegraphics[width=0.4\linewidth]{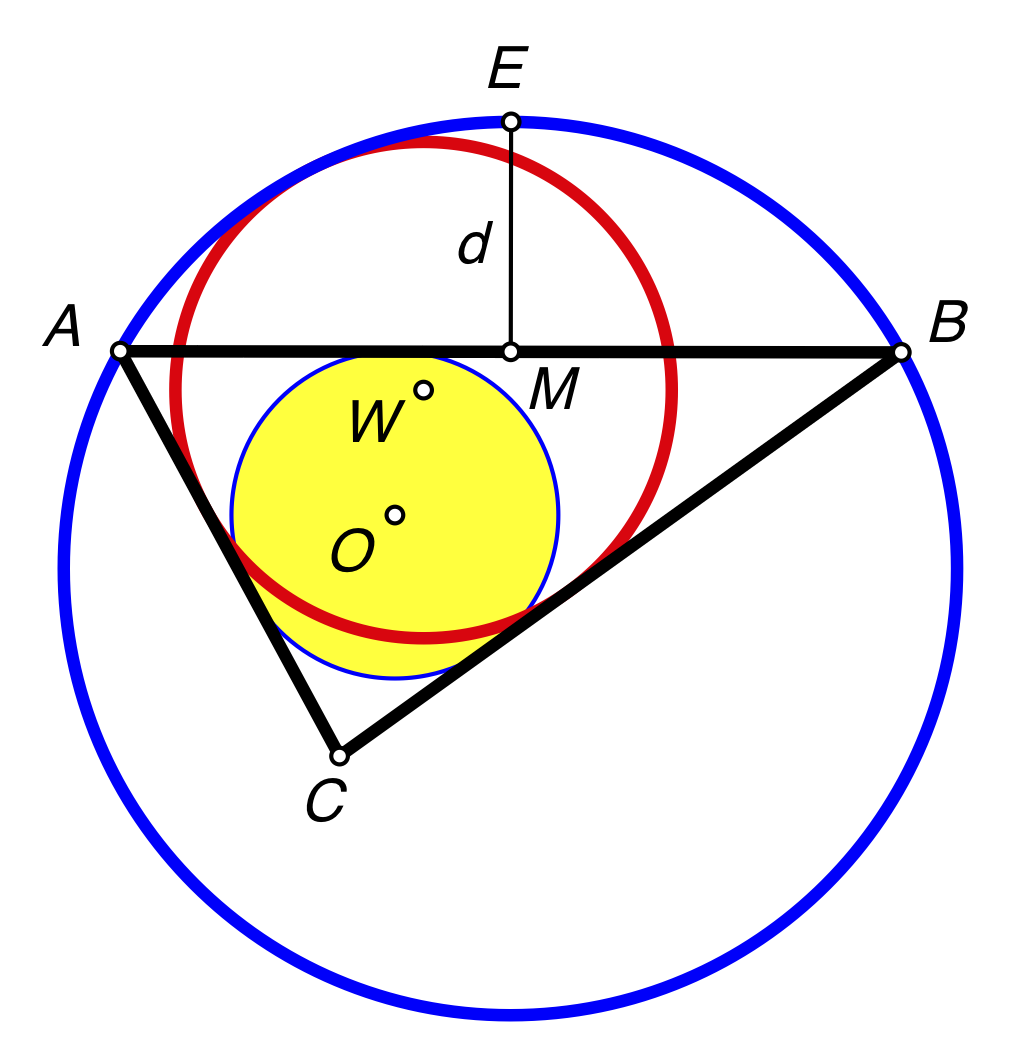}
\caption{$w=r+2d(s-a)(s-b)/(cs)$}
\label{fig:AjimasTheorem}
\end{figure}

\begin{proof}
See \cite[pp.~96-97]{Fukagawa-Rigby}. A more detailed proof can be found in \cite[pp.~40-49]{Unger}.
\end{proof}

\begin{theorem}
\label{thm:4circles}
Let $D$ be any point on side $BC$ of $\triangle ABC$. Cevian $AD$ extended meets the circumcircle of $\triangle ABC$ at $D'$.
Let $W_1(w_1)$ be the incircle of {\Wedge} $BDD'$ and
let $W_2(w_2)$ be the incircle of {\Wedge} $CDD'$.
Let $O_1(r_1)$ be the incircle of $\triangle ADB$ and
let $O_2(r_2)$ be the incircle of $\triangle ADC$
(Figure~\ref{fig:segs-4circles}).
Then
$$\frac{1}{r_1}+\frac{1}{w_2}=\frac{1}{r_2}+\frac{1}{w_1}.$$
\end{theorem}

\begin{figure}[h!t]
\centering
\includegraphics[width=0.4\linewidth]{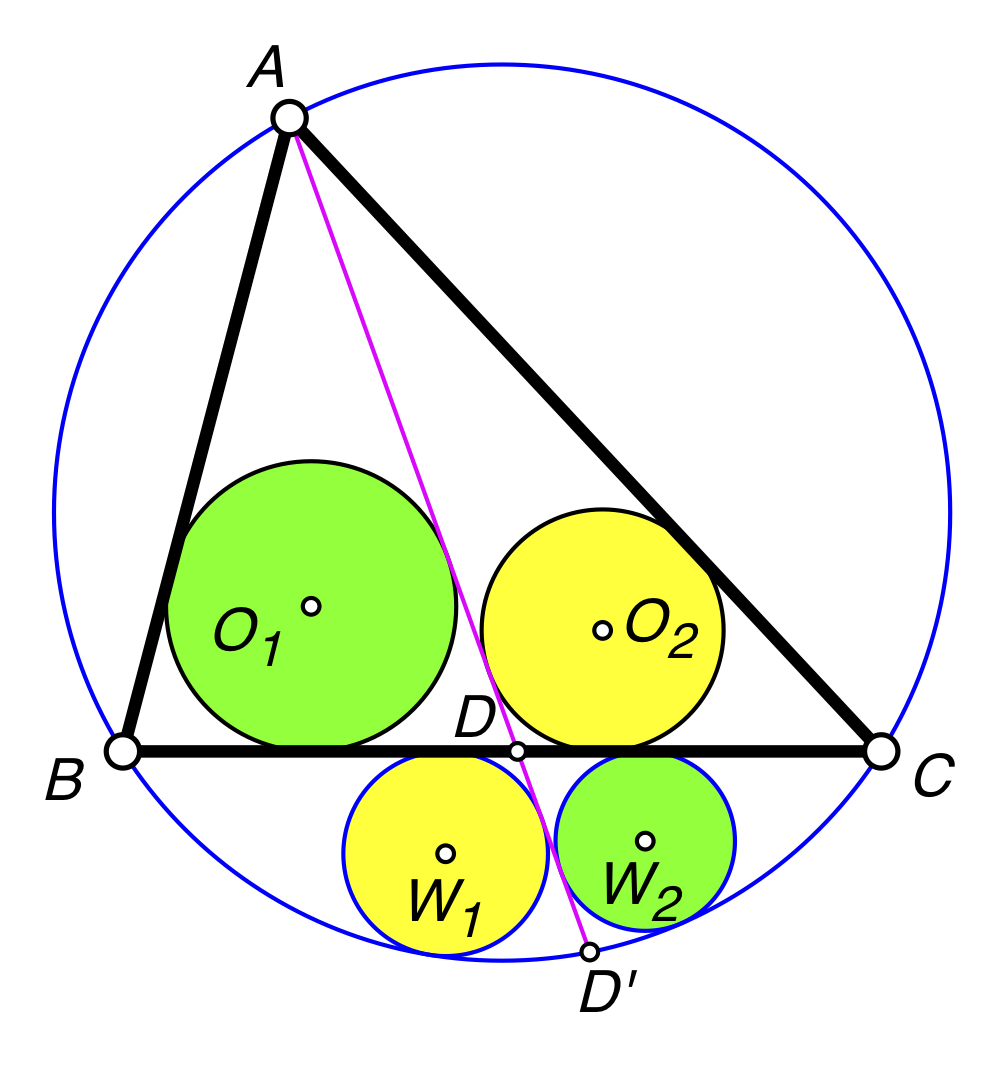}
\caption{$1/r_1+1/w_2=1/r_2+1/w_1$}
\label{fig:segs-4circles}
\end{figure}

\begin{proof}
We give names to the various angles as shown in Figure~\ref{fig:angleNames}.

\begin{figure}[h!t]
\centering
\includegraphics[width=0.5\linewidth]{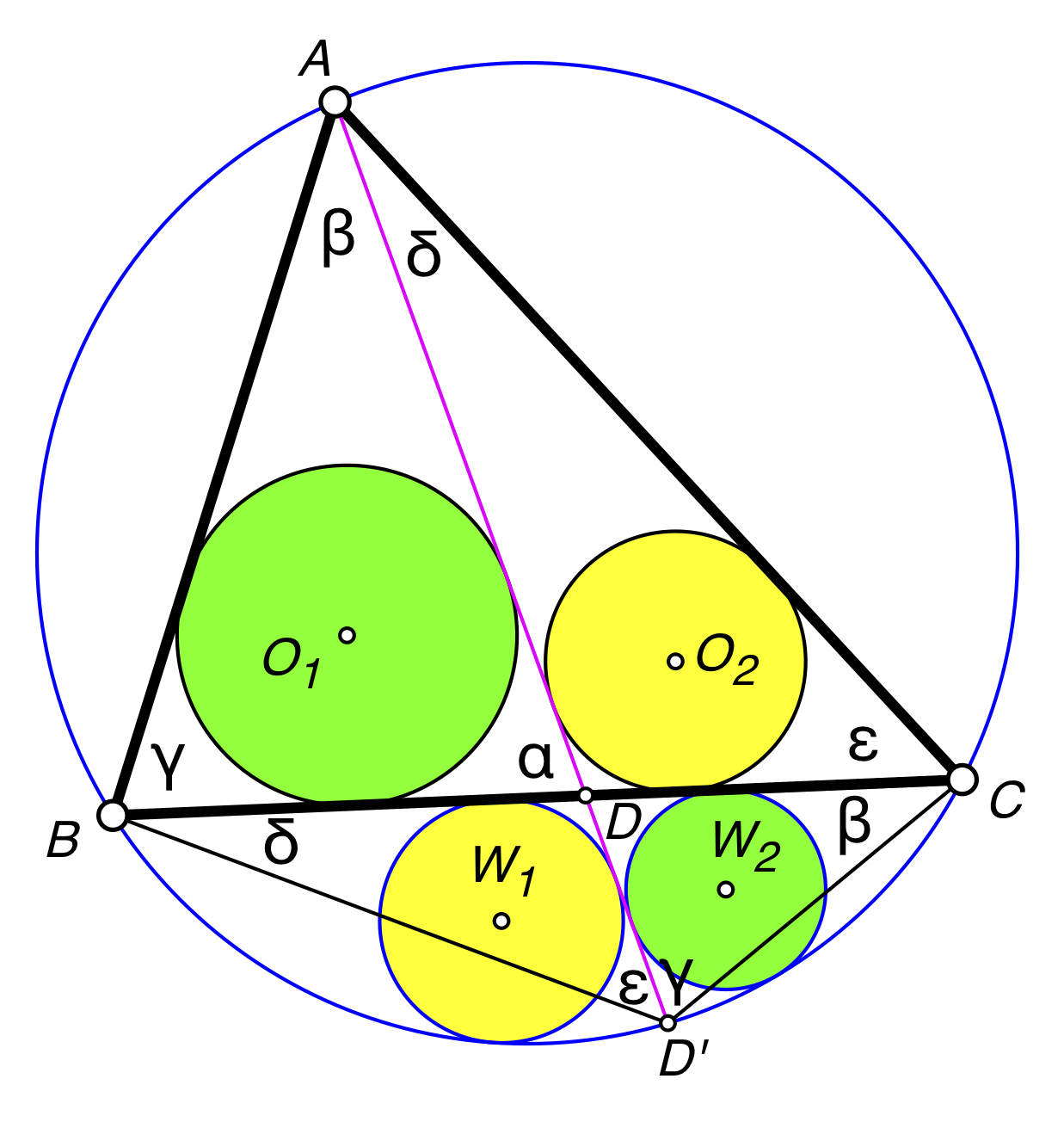}
\caption{angle names}
\label{fig:angleNames}
\end{figure}

Note that $\angle CDA=\alpha_2$ is supplementary to $\angle BDA=\alpha$,
so $\cos(\alpha_2/2)=\sin(\alpha/2)$.
Four applications of Theorem~\ref{thm:wedgeFormula} gives
$$
\begin{aligned}
w_1&=\frac{4R\sin(\beta/2)\cos(\gamma/2)\sin(\delta/2)\sin(\epsilon/2)}{\sin^2(\alpha/2)},\\
w_2&=\frac{4R\sin(\beta/2)\sin(\gamma/2)\sin(\delta/2)\cos(\epsilon/2)}{\cos^2(\alpha/2)},\\
r_1&=\frac{4R\sin(\beta/2)\sin(\gamma/2)\sin(\epsilon/2)\cos(\epsilon/2)}{\cos(\alpha/2)},\\
r_2&=\frac{4R\sin(\delta/2)\sin(\epsilon/2)\sin(\gamma/2)\cos(\gamma/2)}{\cos(\alpha/2)}.
\end{aligned}
$$
Now form the expression
$$S=\frac{1}{w_1}+\frac{1}{r_2}-\left(\frac{1}{w_2}+\frac{1}{r_1}\right).$$
Then substitute $\alpha=\delta+\epsilon$ and then $\epsilon=\pi-\beta-\gamma-\delta$.
Simplifying the resulting expression (using a computer algebra system), shows that $S=0$.
\end{proof}

The result of Theorem~\ref{thm:4circles} is so elegant
that it is unlikely that it is true only because the complicated trigonometric expression, $S$,
in the proof just happens to simplify to 0.

\begin{open}
Is there a simple proof of Theorem \ref{thm:4circles}
that does not involve a large amount of trigonometric computation requiring
computer simplification?
\end{open}

\void{
We also have the following related result.
\begin{theorem}
\label{thm:4moreCircles}
Let $D$ be any point on side $BC$ of $\triangle ABC$. Cevian $AD$ extended meets the circumcircle of $\triangle ABC$ at $D'$.
Let $W_1(w_1)$ be the circle inscribed in {\Wedge} $ADB$ and
let $W_2(w_2)$ be the circle inscribed in {\Wedge} $ADC$.
Let $O_1(r_1)$ be the circle inscribed in $\triangle BDD'$ and
let $O_2(r_2)$ be the circle inscribed in $\triangle CDD'$
(Figure~\ref{fig:4moreCircles}).
Then
$$\frac{1}{r_1}+\frac{1}{w_2}=\frac{1}{r_2}+\frac{1}{w_1}.$$
\end{theorem}
\begin{figure}[h!t]
\centering
\includegraphics[width=0.3\linewidth]{4moreCircles.png}
\caption{$1/r_1+1/w_2=1/r_2+1/w_1$}
\label{fig:4moreCircles}
\end{figure}
\begin{proof}
This is actually the same result as given by Theorem~\ref{thm:4circles} with labels $A$ and $D'$
interchanged.
\end{proof}
}

The following theorem appeared on a tablet in 1844 in the Aichi prefecture.
See \cite[problem~1.4.7]{Fukagawa-Pedoe} and \cite[p.~22]{Fukagawa-2}.

\begin{theorem}
\label{thm:Aichi1844}
Chords $AB$ and $CD$ of a circle meet at $E$. 
Let $W_1(w_1)$ be the incircle of {\Wedge} $BED$ and let $W_2(w_2)$ be the incircle of {\Wedge} $AEC$.
Let $O_1(r_1)$ be the incircle of $\triangle BED$ and let $O_2(r_2)$ be the incircle of $\triangle AEC$.
See Figure~\ref{fig:Aichi1844}.
Then
$$\frac{1}{r_1}+\frac{1}{w_2}=\frac{1}{r_2}+\frac{1}{w_1}.$$
\end{theorem}

\begin{figure}[h!t]
\centering
\includegraphics[width=0.4\linewidth]{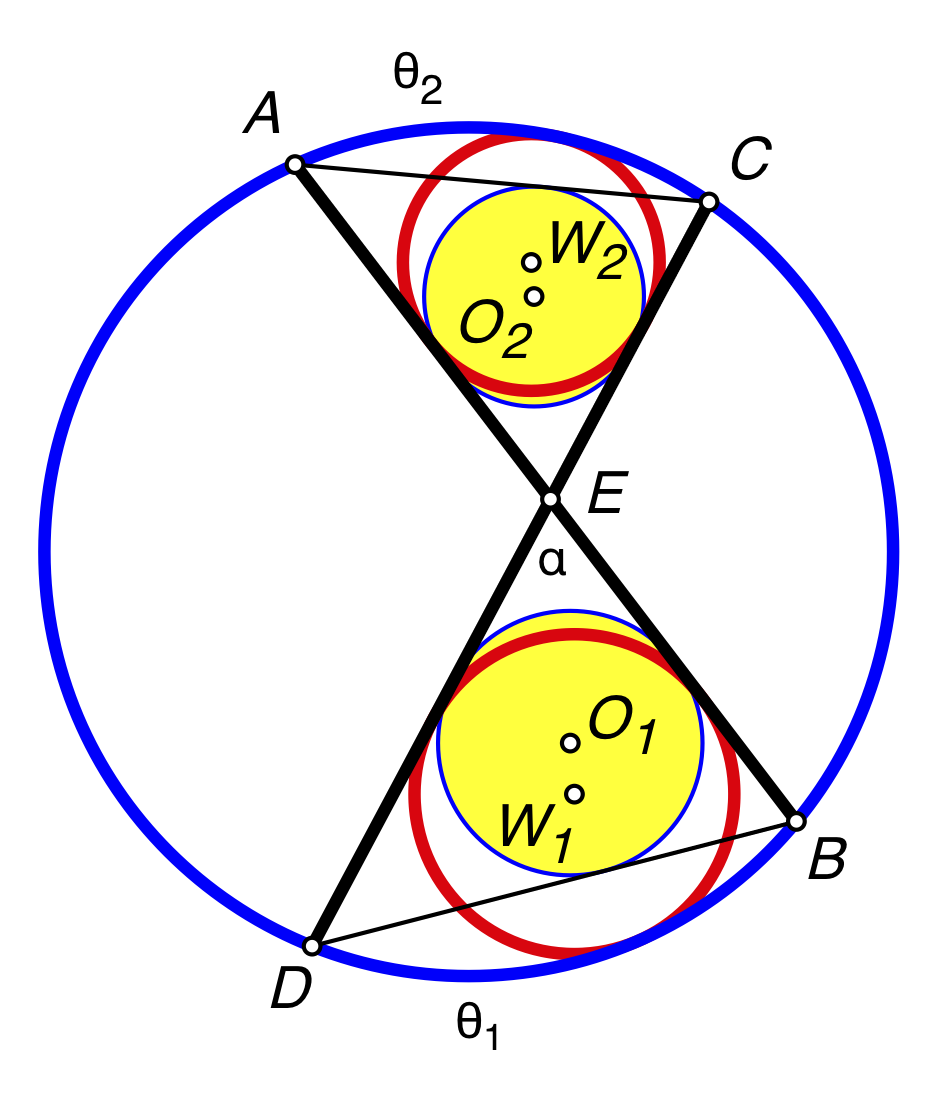}
\caption{$1/r_1+1/w_2=1/r_2+1/w_1$}
\label{fig:Aichi1844}
\end{figure}

\begin{proof}
This proof comes from \cite[pp.~26--27]{Unger30}.
Let $\alpha$ be the vertex angle of {\Wedge} $BED$. Let $\theta_1$ be its arc angle an let $\theta_2$
be the arc angle of the opposite {\Wedge} $AEC$.
By Theorem~\ref{thm:wrFormula2}, we have
$$\frac{w_1}{r_1}=1+\tan(\alpha/2)\tan(\theta_1/4)$$
which is equivalent to
$$\frac{w_1}{r_1}-1=\tan(\alpha/2)\tan(\theta_1/4)$$
or
$$\frac{1}{r_1}-\frac{1}{w_1}=\frac{\tan(\alpha/2)\tan(\theta_1/4)}{w_1}.$$
Similarly,
$$\frac{1}{r_2}-\frac{1}{w_2}=\frac{\tan(\alpha/2)\tan(\theta_2/4)}{w_2}.$$
But $$\frac{\tan(\theta_1/4)}{w_1}=\frac{\tan(\theta_2/4)}{w_2}$$
by Theorem~\ref{thm:wwFormula} (which will be proved in the next section). Thus,
$$\frac{1}{r_1}-\frac{1}{w_1}=\frac{1}{r_2}-\frac{1}{w_2}$$
and the result follows.
\end{proof}

\begin{theorem}
\label{thm:crossedChords}
Chords $AB$ and $CD$ of a circle meet at $E$, with $\angle AEC=\alpha$. 
Let $W_1(w_1)$ be the incircle of {\Wedge} $BED$ and let $W_2(w_2)$ be the incircle of {\Wedge} $AEC$.
Let $O_1(r_1)$ be the incircle of $\triangle BED$ and let $O_2(r_2)$ be the incircle of $\triangle AEC$.
See Figure~\ref{fig:Aichi1844}.
Then
$$r_1r_2=w_1w_2\cos^2\frac{\alpha}{2}.$$
\end{theorem}

\begin{proof}
Let $m(\arc{D}{B})=\theta_1$ and $m(\arc{C}{A})=\theta_2$.
Applying Theorem~\ref{thm:wrFormula} to {\Wedge} $BED$ gives
$$\frac{w_1}{r_1}=\frac{\cos(\theta_2/4)}{\cos(\alpha/2)\cos(\theta_1/4)}.$$
Applying Theorem~\ref{thm:wrFormula} to {\Wedge} $AEC$ gives
$$\frac{w_2}{r_2}=\frac{\cos(\theta_1/4)}{\cos(\alpha/2)\cos(\theta_2/4)}.$$
Multiplying these two equations 
gives
$$\frac{w_1w_2}{r_1r_2}=\frac{1}{\cos^2(\alpha/2)}$$
and the result follows by cross-multiplying.
\end{proof}

\begin{theorem}
\label{thm:4circlesEqualAngles}
Cevians $AD$ and $CF$ in $\triangle ABC$ meet at $P$ and $\angle BFC=\angle BDA$. The cevians are extended to meet the circumcircle of $\triangle ABC$ at points $D'$ and $F'$, respectively,
as shown in Figure~\ref{fig:4circlesEqualAngles}.
Let $W_1(w_1)$ be the incircle of {\Wedge} $BDD'$ and let $W_2(w_2)$ be the incircle of {\Wedge} $BFF'$.
Let $O_1(r_1)$ be the incircle of $\triangle BDD'$ and let $O_2(r_2)$ be the incircle of $\triangle BFF'$.
Then
$$\frac{w_1}{r_1}=\frac{w_2}{r_2}.$$
\end{theorem}

\begin{figure}[h!t]
\centering
\includegraphics[width=0.5\linewidth]{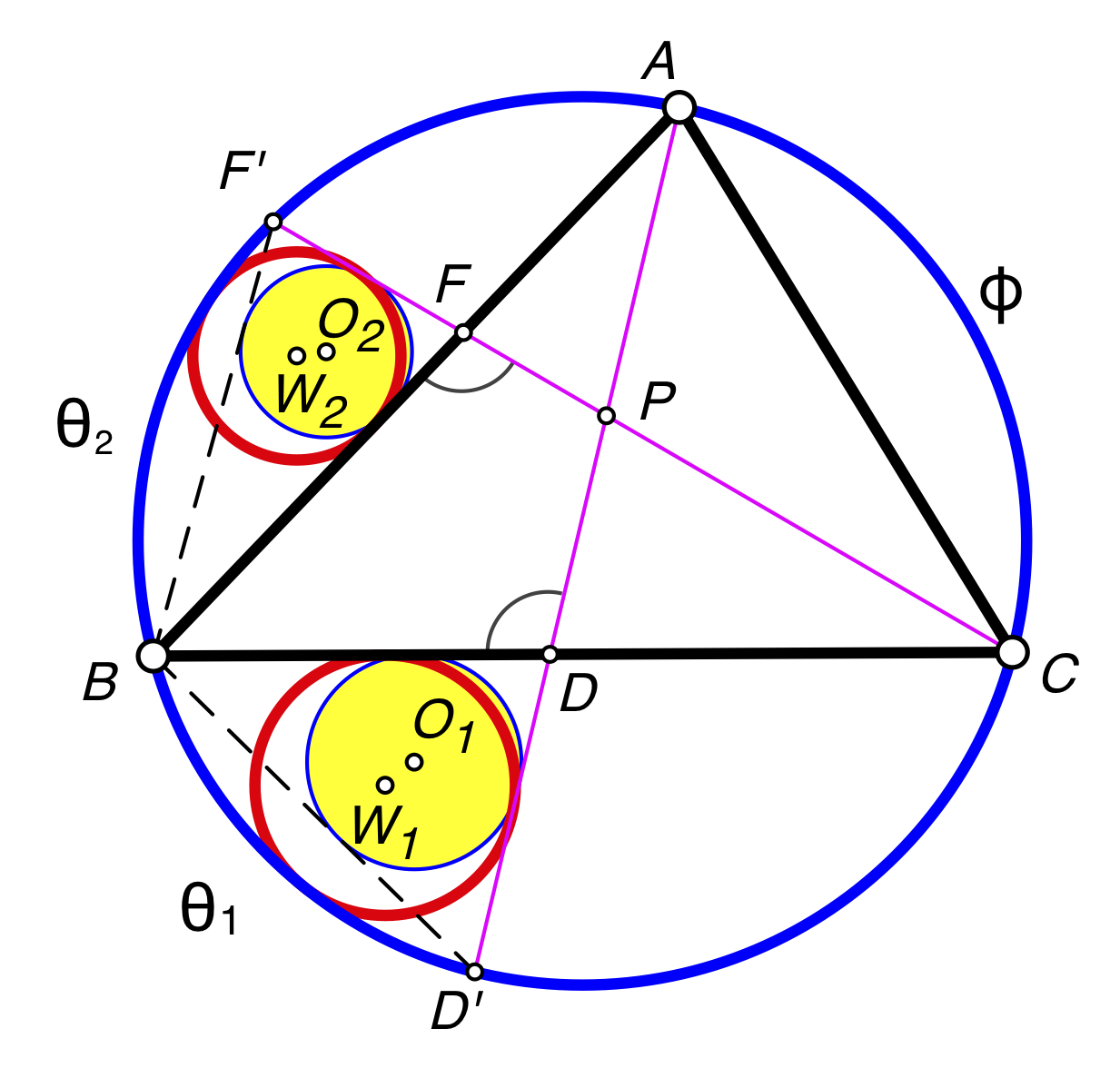}
\caption{$w_1/r_1=w_2/r_2$}
\label{fig:4circlesEqualAngles}
\end{figure}

\begin{proof}
Let $m(\arc{B}{D'})=\theta_1$, $m(\arc{F'}{B})=\theta_2$, and $m(\arc{C}{A})=\phi$.
Applying Theorem~\ref{thm:wrFormula} to {\Wedge} $BDD'$ using $\alpha_1=\angle BDD'$ gives
$$\frac{w_1}{r_1}=\frac{\cos(\phi/4)}{\cos(\alpha_1/2)\cos(\theta_1/4)}.$$
Applying Theorem~\ref{thm:wrFormula} to {\Wedge} $BFF'$ using $\alpha_2=\angle BFF'$ gives
$$\frac{w_2}{r_2}=\frac{\cos(\phi/4)}{\cos(\alpha_2/2)\cos(\theta_2/4)}.$$
But $\alpha_1=\alpha_2$ because they are supplementary to the two given angles.
Chords $AD'$ and $BC$ intercept arcs of measures $\theta_1$ and $\phi$,
so $\alpha_1=(\theta_1+\phi)/2$. Similarly, $\alpha_2=(\theta_2+\phi)/2$.
Thus, $\theta_1=\theta_2$ because $\alpha_1=\alpha_2$.
Therefore, $w_1/r_1=w_2/r_2$.
\end{proof}

\begin{theorem}
\label{thm:segs-4Hcircles}
Let $H$ be the orthocenter of acute triangle $ABC$. The altitudes $AD$ and $CF$ are extended to meet the circumcircle of $\triangle ABC$ at points $D'$ and $F'$, respectively.
Let $W_1(w_1)$ be the incircle of {\Wedge} $BDD'$ and let $W_2(w_2)$ be the incircle of {\Wedge} $BFF'$.
Let $O_1(r_1)$ be the incircle of $\triangle BDH$ and let $O_2(r_2)$ be the incircle of $\triangle BFH$.
See Figure~\ref{fig:segs-4Hcircles}.
Then
$$\frac{w_1}{r_1}=\frac{w_2}{r_2}.$$
\end{theorem}

\begin{figure}[h!t]
\centering
\includegraphics[width=0.45\linewidth]{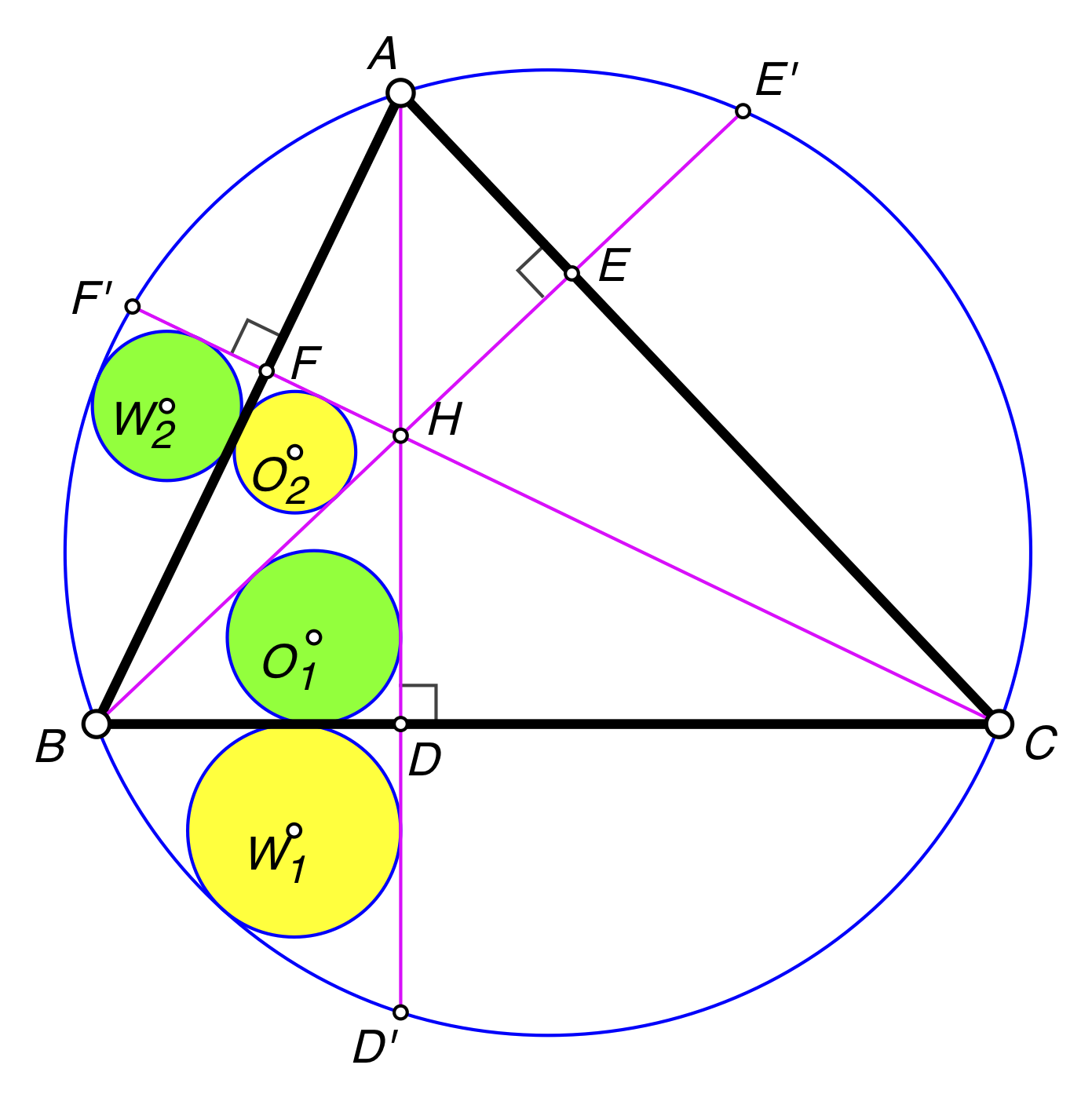}
\caption{$w_1/r_1=w_2/r_2$}
\label{fig:segs-4Hcircles}
\end{figure}

\begin{proof}
Let $r'_1$ be the inradius of $\triangle BDD'$ and let $r'_2$ be the inradius of $\triangle BFF'$.
Since $\angle BFC=\angle BDA$, by Theorem~\ref{thm:4circlesEqualAngles}, we have
$w_1/r'_1=w_2/r'_2$.
Now $\angle CBD'=\angle CAD'$ since both angles subtend the same arc.
But $\angle CAD'=\angle CBE'$ since both angles are complementary to $\angle ACB$.
Thus, $\angle DBD'=\angle DBH$. Right triangles $BDD'$ and $BDH$ share a common side.
Thus $\triangle BDD'\cong\triangle BDH$. Hence $r_1=r'_1$. Similarly, $r_2=r'_2$.
Therefore, $w_1/r_1=w_2/r_2$.
\end{proof}

\begin{lemma}
\label{lemma:4Hcircles}
Let $H$ be the orthocenter of acute $\triangle ABC$, and let the altitudes be $AD$, $BE$, and $CF$ as shown in Figure~\ref{fig:4Hcircles}. Circles $O_1(r_1)$, $O_2(r_2)$, $O_3(r_3)$, and $O_4(r_4)$ are inscribed in triangles $BHD$, $BHF$, $CAF$, and $ACD$, respectively.
Then
$r_1/r_2=r_4/r_3$.
\end{lemma}

\begin{figure}[h!t]
\centering
\includegraphics[width=0.5\linewidth]{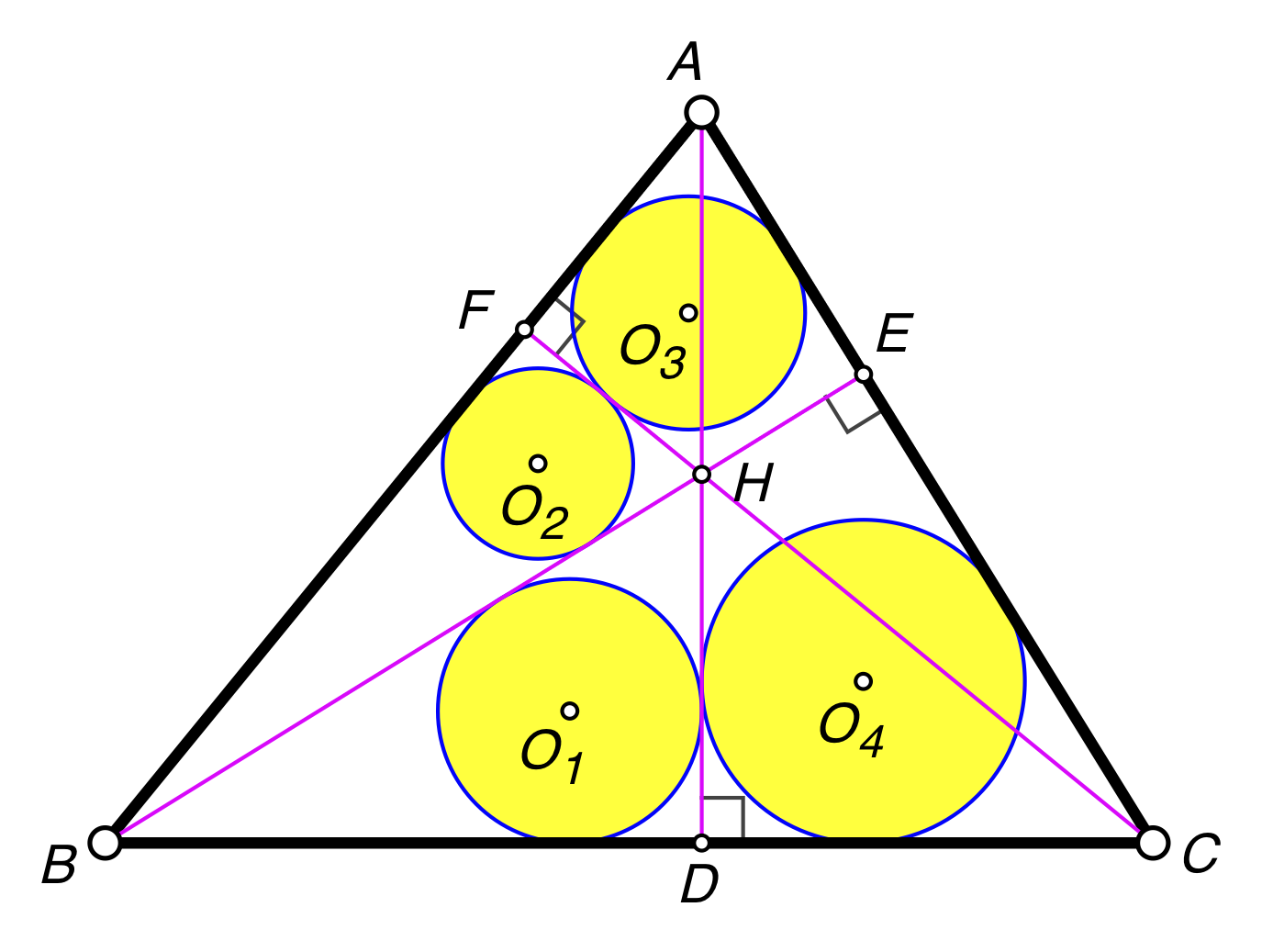}
\caption{$r_1/r_2=r_4/r_3$}
\label{fig:4Hcircles}
\end{figure}

\begin{proof}
Note that $\triangle BHF\sim\triangle CAF$. Therefore the figure consisting
of $\triangle BHF$ and its incircle is similar to the figure consisting of $\triangle CAF$
and its incircle. Corresponding parts of similar figures are in proportion, so $r_2/r_3=BH/CA$.
In the same manner, $\triangle BHD\sim\triangle ACD$ which implies that $r_1/r_4=BH/AC$.
Therefore, $r_2/r_3=r_1/r_4$ or $r_1/r_2=r_4/r_3$.
\end{proof}

\begin{corollary}
\label{cor:segs-4H}
Let $H$ be the orthocenter of acute $\triangle ABC$. The altitudes $AD$ and $CF$ are extended to meet the circumcircle of $\triangle ABC$ at points $D'$ and $F'$, respectively.
Let $W_1(w_1)$ be the incircle of {\Wedge} $BDD'$ and let $W_2(w_2)$ be the incircle of {\Wedge} $BFF'$.
Let $O_3(r_3)$ be the incircle of $\triangle AFC$ and let $O_4(r_4)$ be the incircle of $\triangle ADC$.
See Figure~\ref{fig:segs-4H}.
Then
$w_1/w_2=r_4/r_3$.
\end{corollary}

\begin{figure}[h!t]
\centering
\includegraphics[width=0.5\linewidth]{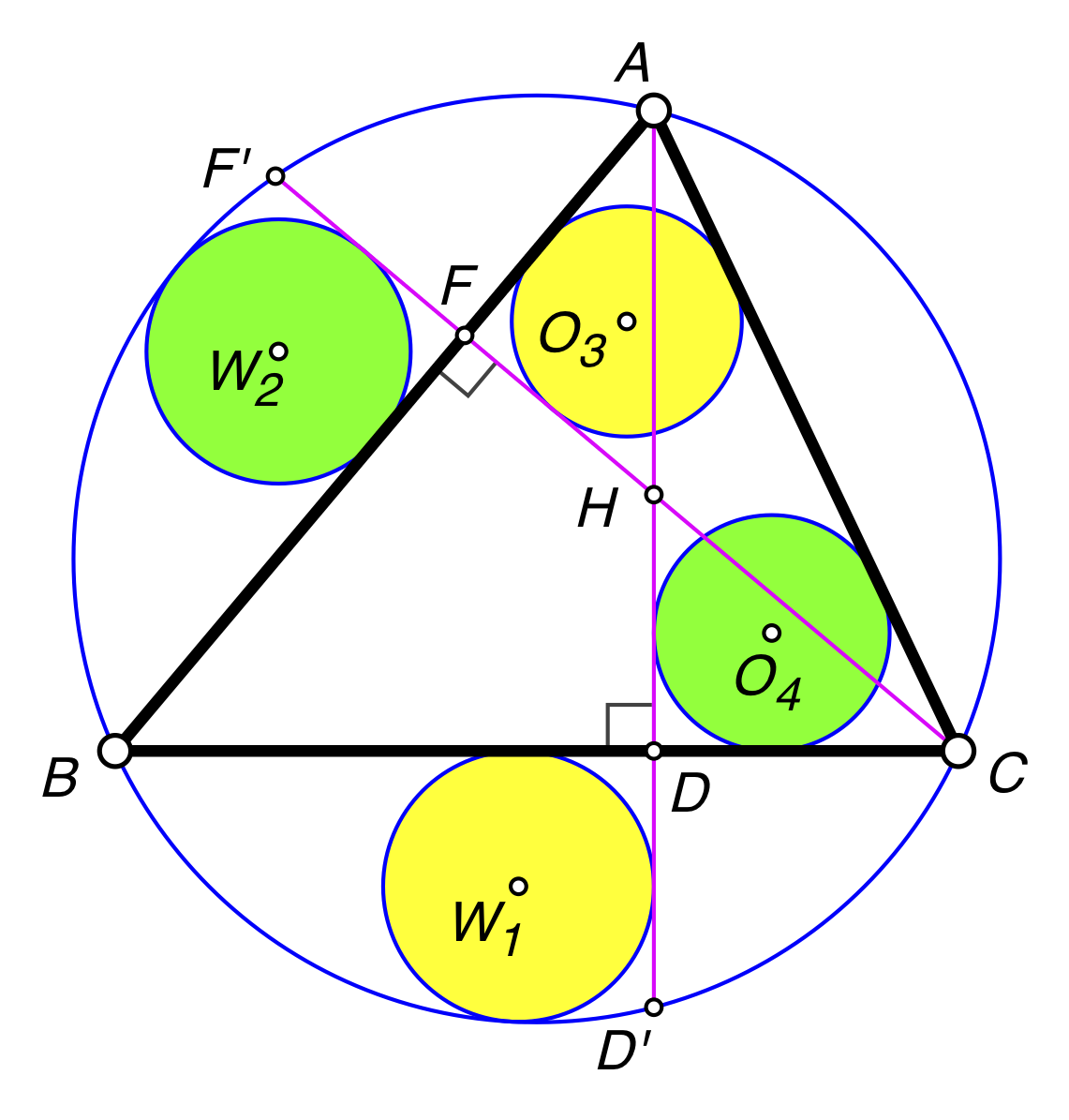}
\caption{$w_1/w_2=r_4/r_3$}
\label{fig:segs-4H}
\end{figure}

\begin{proof}
By Theorem \ref{thm:segs-4Hcircles}, $w_1/w_2=r_1/r_2$. By Lemma~\ref{lemma:4Hcircles},
$r_1/r_2=r_4/r_3$. Therefore, $w_1/w_2=r_4/r_3$.
\end{proof}

\section{Relationships Between the Incircles of Two Skewed Sectors}

\begin{theorem}
\label{thm:Aichi1844trig}
Chords $AB$ and $CD$ of a circle meet at $E$. 
Let $W_1(w_1)$ be the circle inscribed in {\Wedge} $DEB$ and let $W_2(w_2)$ be circle inscribed in {\Wedge} $AEC$, as shown in Figure~\ref{fig:Aichi1844trig}.
Then
$$\frac{w_1}{w_2}=\frac{\tan(\theta_1/2)}{\tan(\theta_2/2)}$$
where $\angle BCD=\theta_1$ and $\angle ADC=\theta_2$.
\end{theorem}

\begin{figure}[h!t]
\centering
\includegraphics[width=0.25\linewidth]{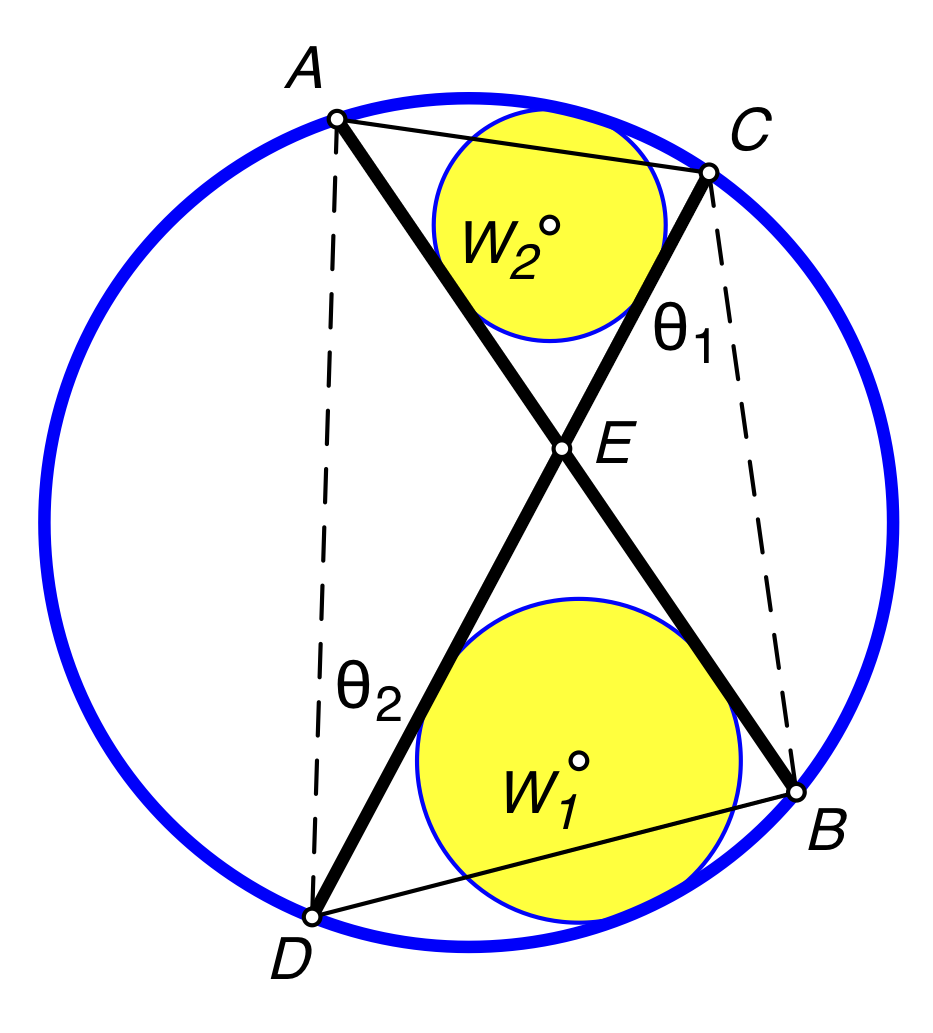}
\caption{$w_1/w_2=\tan(\theta_1/2)/\tan(\theta_2/2)$}
\label{fig:Aichi1844trig}
\end{figure}

\begin{proof}
See \cite[pp.~96--97]{Fukagawa-Rigby} or \cite[p.~26--27]{Unger30}.
\end{proof}

Since the measure of an angle inscribed in a circle is half the circular measure of the intercepted arc, we have the following result.

\begin{theorem}
\label{thm:wwFormula}
Chords $AB$ and $CD$ of a circle meet at $E$. 
Let $W_1(w_1)$ be the circle inscribed in {\Wedge} $DEB$ and let $W_2(w_2)$ be the circle inscribed in {\Wedge} $AEC$, as shown in Figure~\ref{fig:wwFormula}.
Then
$$\frac{w_1}{w_2}=\frac{\tan(\theta_1/4)}{\tan(\theta_2/4)}$$
where $m(\arc{C}{A})=\theta_1$ and $m(\arc{D}{B})=\theta_2$.
\end{theorem}

\begin{figure}[h!t]
\centering
\includegraphics[width=0.3\linewidth]{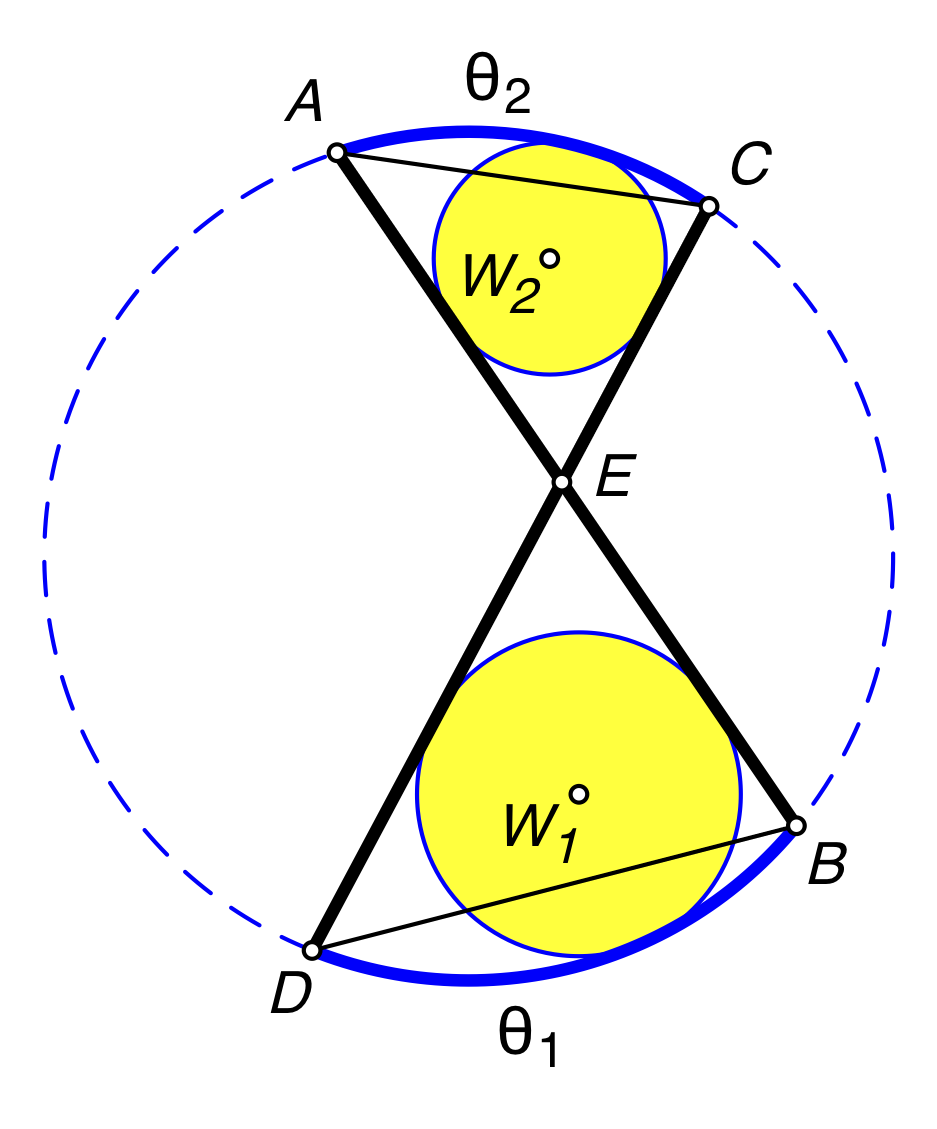}
\caption{$w_1/w_2=\tan(\theta_1/4)/\tan(\theta_2/4)$}
\label{fig:wwFormula}
\end{figure}

The reader may wonder if there is any geometric significance to the angle $\theta_1/4$. If $M$ is the midpoint
of arc $\arc{D}{B}$, then $\angle BDM=\frac{1}{2}\arc{M}{B}$ and $\arc{M}{B}=\frac{1}{2}\arc{D}{B}=\frac{1}{2}\theta_1$,
so $\angle BDM=\theta_1/4$.

There is a related result involving the incircles of triangles.

\begin{theorem}
\label{thm:rrFormula}
Chords $B_1B_2$ and $C_1C_2$ of a circle meet at $A$. 
Let $r_1$ and $r_2$ be the inradii of $\triangle B_1AC_1$ and $\triangle B_2AC_2$, respectively,
as shown in Figure~\ref{fig:rrFormula}.
Let $B_1C_1=a_1$ and let $B_2C_2=a_2$.
Then
$=r_1/r_2=a_1/a_2$.
\end{theorem}

\begin{figure}[h!t]
\centering
\includegraphics[width=0.3\linewidth]{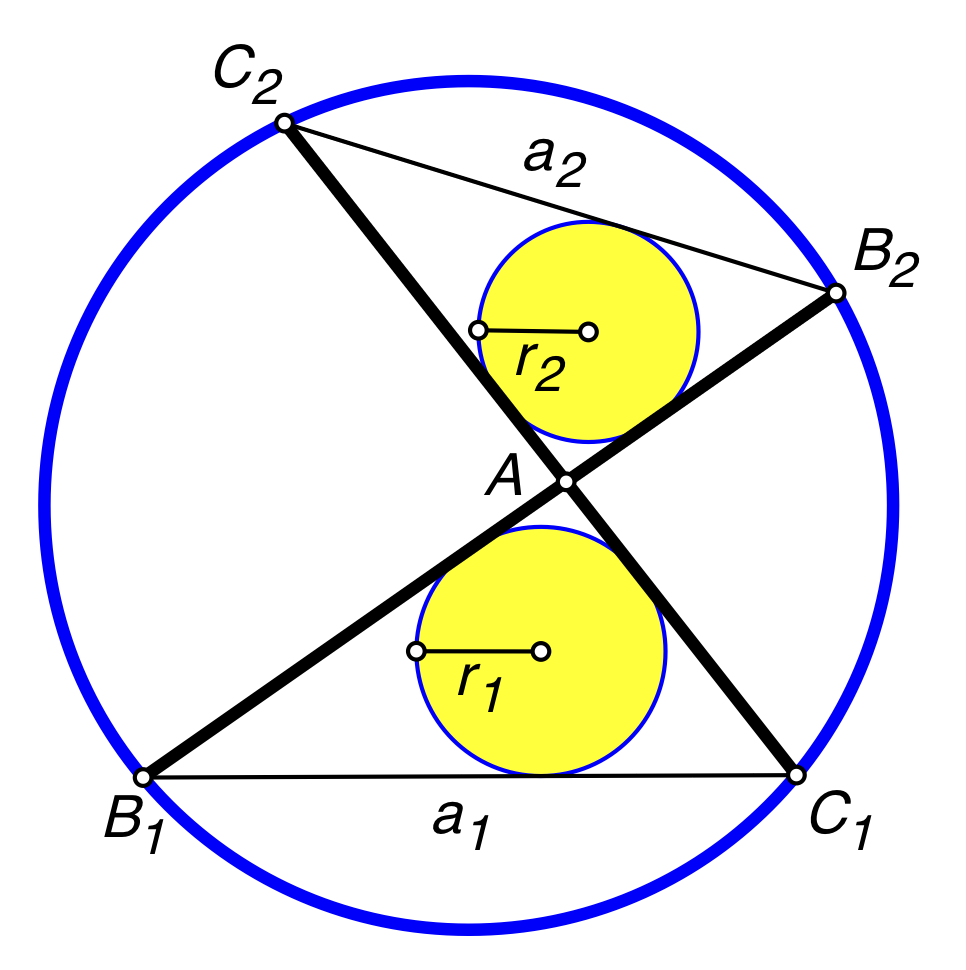}
\caption{$r_1/r_2=a_1/a_2$}
\label{fig:rrFormula}
\end{figure}

\begin{proof}
This follows from the fact that $\triangle B_1AC_1\sim\triangle B_2AC_2$.
\end{proof}

The following theorem comes from \cite[Problem~21]{Unger30} and is related to Ajima's Theorem.

\begin{theorem}
\label{thm:sagittae}
Chords $B_1B_2$ and $C_1C_2$ of a circle meet at $A$. 
Let $W_1(w_1)$ be the circle inscribed in {\Wedge} $B_1AC_1$ and let $W_2(w_2)$ be the circle inscribed in {\Wedge} $B_2AC_2$. Let $v_1$ and $v_2$ be the heights of the segments formed by chords $B_1C_1$ and $B_2C_2$ as shown in Figure~\ref{fig:sagittae}.
Then
$$\frac{w_1}{w_2}=\frac{v_1a_2}{v_2a_1}$$
where $a_1=B_1C_1$ and $a_2=B_2C_2$.
\end{theorem}

\begin{figure}[h!t]
\centering
\includegraphics[width=0.34\linewidth]{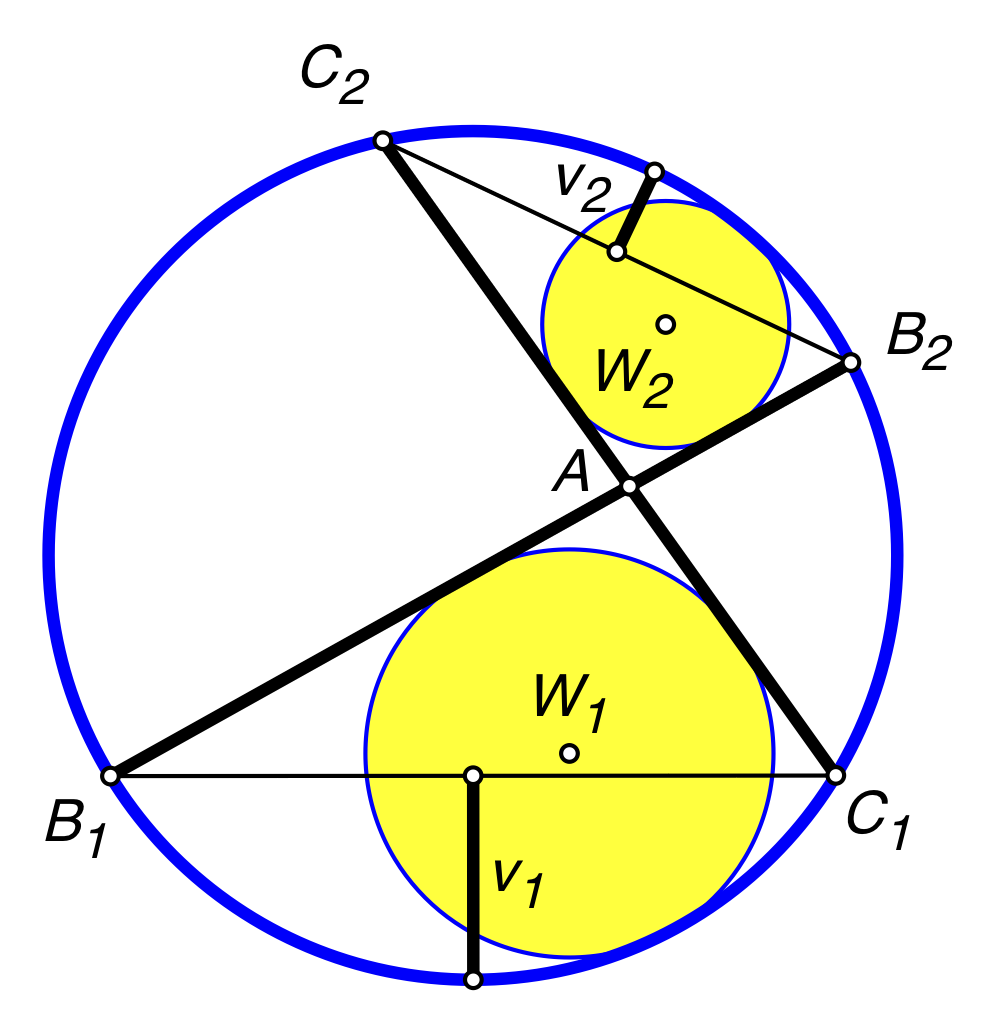}
\caption{$w_1/w_2=v_1a_2/v_2a_1$}
\label{fig:sagittae}
\end{figure}

The following result is due to Pohoatza and Ehrmann, \cite{Pohoatza}.

\begin{theorem}
\label{thm:side-Na}
Let $D$ be the point on side $BC$ of $\triangle ABC$ such that $AB+BD=AC+CD$.
A circle is circumscribed about $\triangle ABC$.
Let $W_1(w_1)$ be the circle inscribed in {\Wedge} $ADB$ and
let $W_2(w_2)$ be the circle inscribed in {\Wedge} $ADC$ (Figure~\ref{fig:side-Na}). 
Then $w_1=w_2$.
\end{theorem}

\begin{figure}[h!t]
\centering
\includegraphics[width=0.35\linewidth]{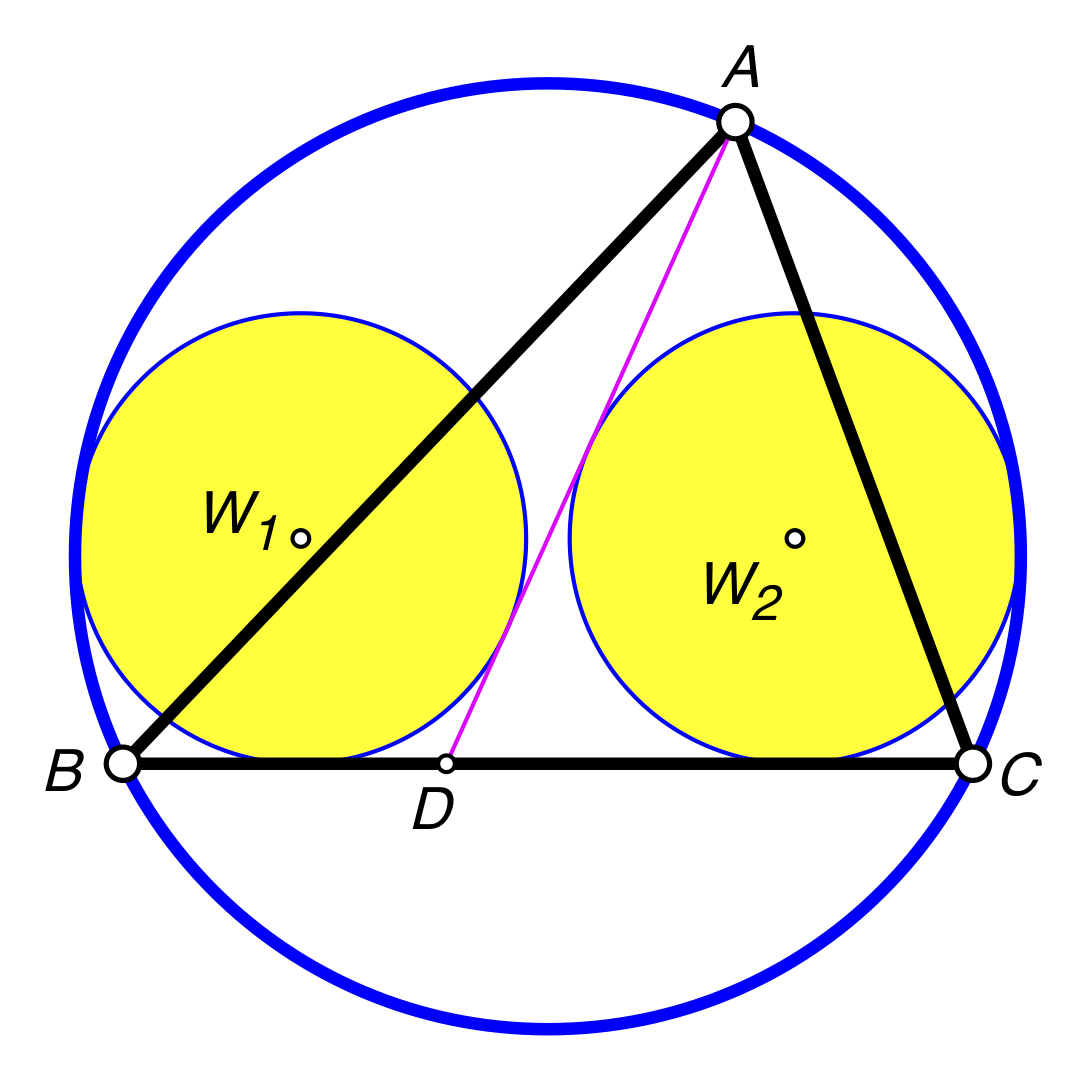}
\caption{$w_1=w_2$}
\label{fig:side-Na}
\end{figure}

\begin{proof}
Extend $AD$ to meet the circumcircle of $\triangle ABC$ at $D'$.
Let $O_1(r_1)$ be the circle inscribed in $\triangle BDD'$ and
let $O_2(r_2)$ be the circle inscribed in $\triangle CDD'$ (Figure~\ref{fig:side-Na-proof}). 
Then $1/r_1+1/w_2=1/r_2+1/w_1$ by Theorem~\ref{thm:4circles} (with points $A$ and $D'$ interchanged).
But $r_1=r_2$ by Theorem~3.4 of \cite{Rabinowitz-other}.
Therefore, $w_1=w_2$.
\end{proof}

See \cite{Ayme} for another proof.

\begin{figure}[h!t]
\centering
\includegraphics[width=0.45\linewidth]{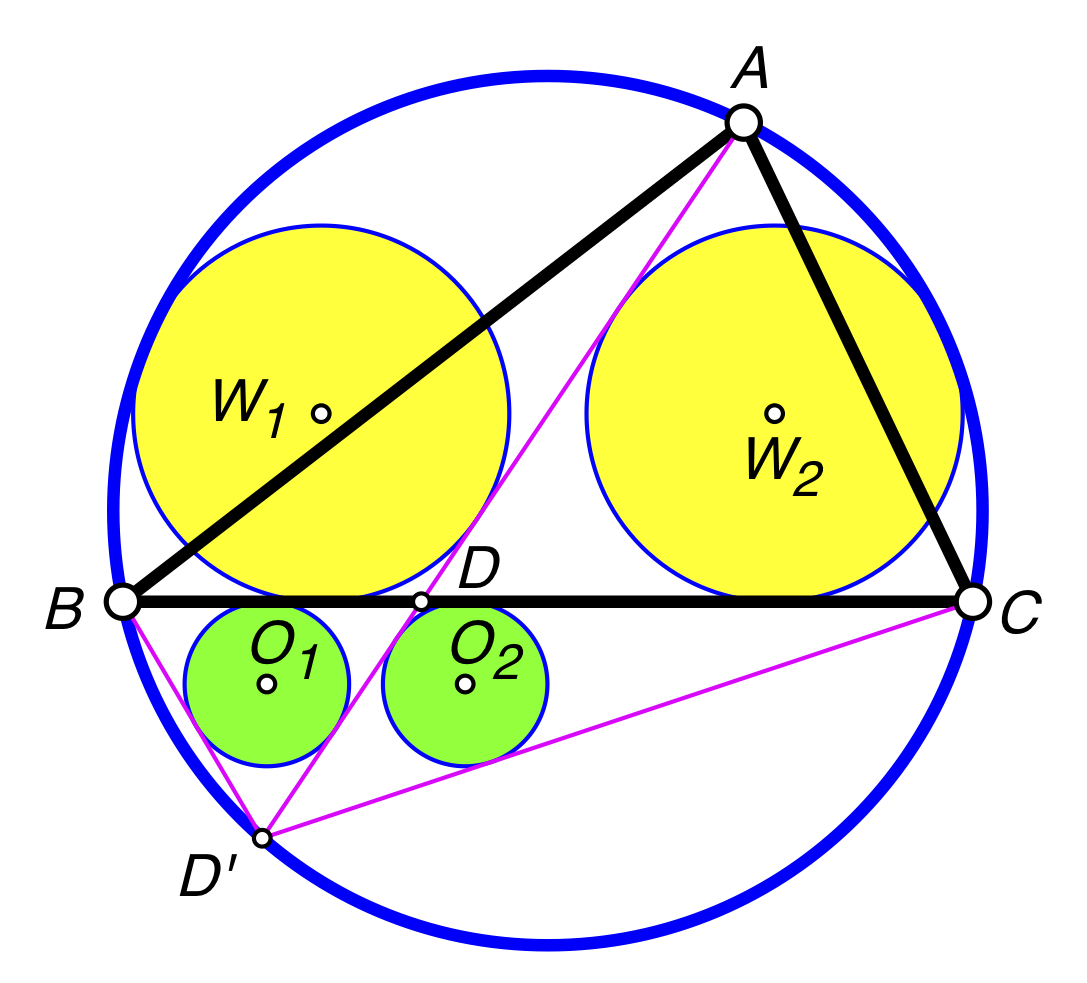}
\caption{$r_1=r_2$ and $w_1=w_2$}
\label{fig:side-Na-proof}
\end{figure}

\section{Relationships Between the Incircles of Six Skewed Sectors}

\begin{theorem}
\label{thm:mixti-H}
Let $H$ be the orthocenter of acute $\triangle ABC$. The altitudes through $H$ extended to meet the circumcircle of $\triangle ABC$ divide the interior of that circumcircle into six \Wedges,
each with vertex at $H$,
as shown in Figure~\ref{fig:mixti-H}. 
Let $W_i(w_i)$ be the circle tangent to two altitudes and internally tangent to
the circumcircle as shown.
Then $w_1w_3w_5=w_2w_4w_6$.
\end{theorem}

\begin{figure}[h!t]
\centering
\includegraphics[width=0.4\linewidth]{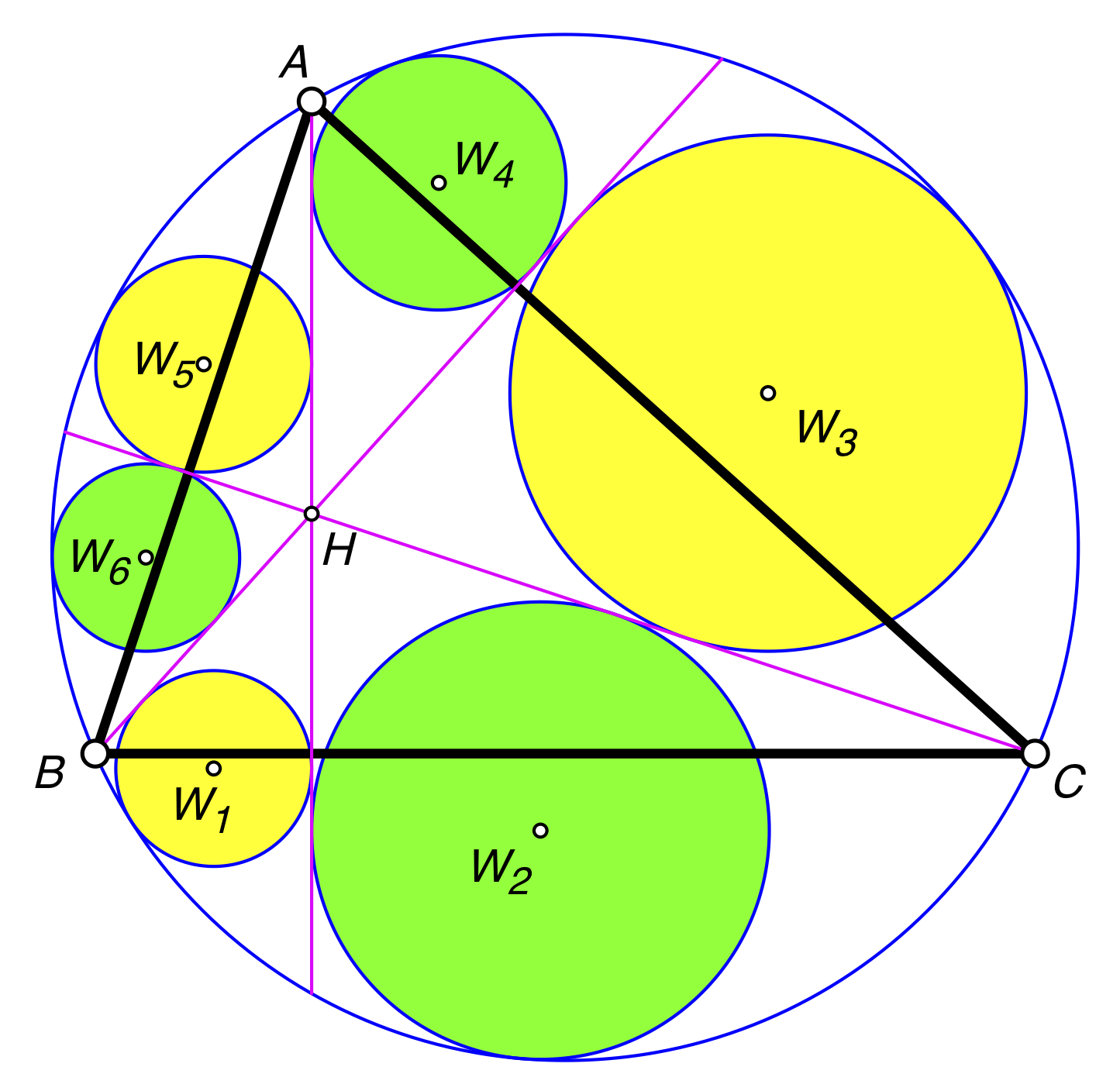}
\caption{$w_1w_3w_5=w_2w_4w_6$}
\label{fig:mixti-H}
\end{figure}

\begin{proof}
Let $\theta_i$ be the arc angle of the {\Wedge} containing circle $W_i$.
By Theorem~\ref{thm:wwFormula}, $w_1/w_4=\tan(\theta_1/4)/\tan(\theta_4/4)$.
Similarly, $w_2/w_5=\tan(\theta_2/4)/\tan(\theta_5/4)$ and $w_3/w_6=\tan(\theta_3/4)/\tan(\theta_6/4)$.
Consequently,
$$\frac{w_1w_3w_5}{w_2w_4w_6}=\frac{w_1}{w_4}\cdot\frac{w_3}{w_6}\cdot\frac{w_5}{w_2}
=\frac{\tan(\theta_1/4)}{\tan(\theta_4/4)}\cdot\frac{\tan(\theta_3/4)}{\tan(\theta_6/4)}\cdot\frac{\tan(\theta_5/4)}{\tan(\theta_2/4)}.$$
Note that $\angle BAH=\angle BCH$ since both are complementary to $\angle ABC$.
Therefore, $\theta_1=\theta_6$. Similarly, $\theta_2=\theta_3$ and $\theta_4=\theta_5$.
Hence
$$\frac{w_1w_3w_5}{w_2w_4w_6}=\frac{\tan(\theta_1/4)}{\tan(\theta_5/4)}\cdot\frac{\tan(\theta_3/4)}{\tan(\theta_1/4)}\cdot\frac{\tan(\theta_5/4)}{\tan(\theta_3/4)}=1,$$
so $w_1w_3w_5=w_2w_4w_6$.
\end{proof}

\begin{theorem}
\label{thm:mixti-I}
Let $I$ be the incenter of $\triangle ABC$. The cevians through $I$ extended to meet the circumcircle of $\triangle ABC$ divide the interior of that circumcircle into six \Wedges,
each with vertex at $I$,
as shown in Figure~\ref{fig:mixti-I}. 
Let $W_i(w_i)$ be the circle tangent to two cevians and internally tangent to
the circumcircle as shown.
Then $w_1w_3w_5=w_2w_4w_6$.
\end{theorem}

\begin{figure}[h!t]
\centering
\includegraphics[width=0.45\linewidth]{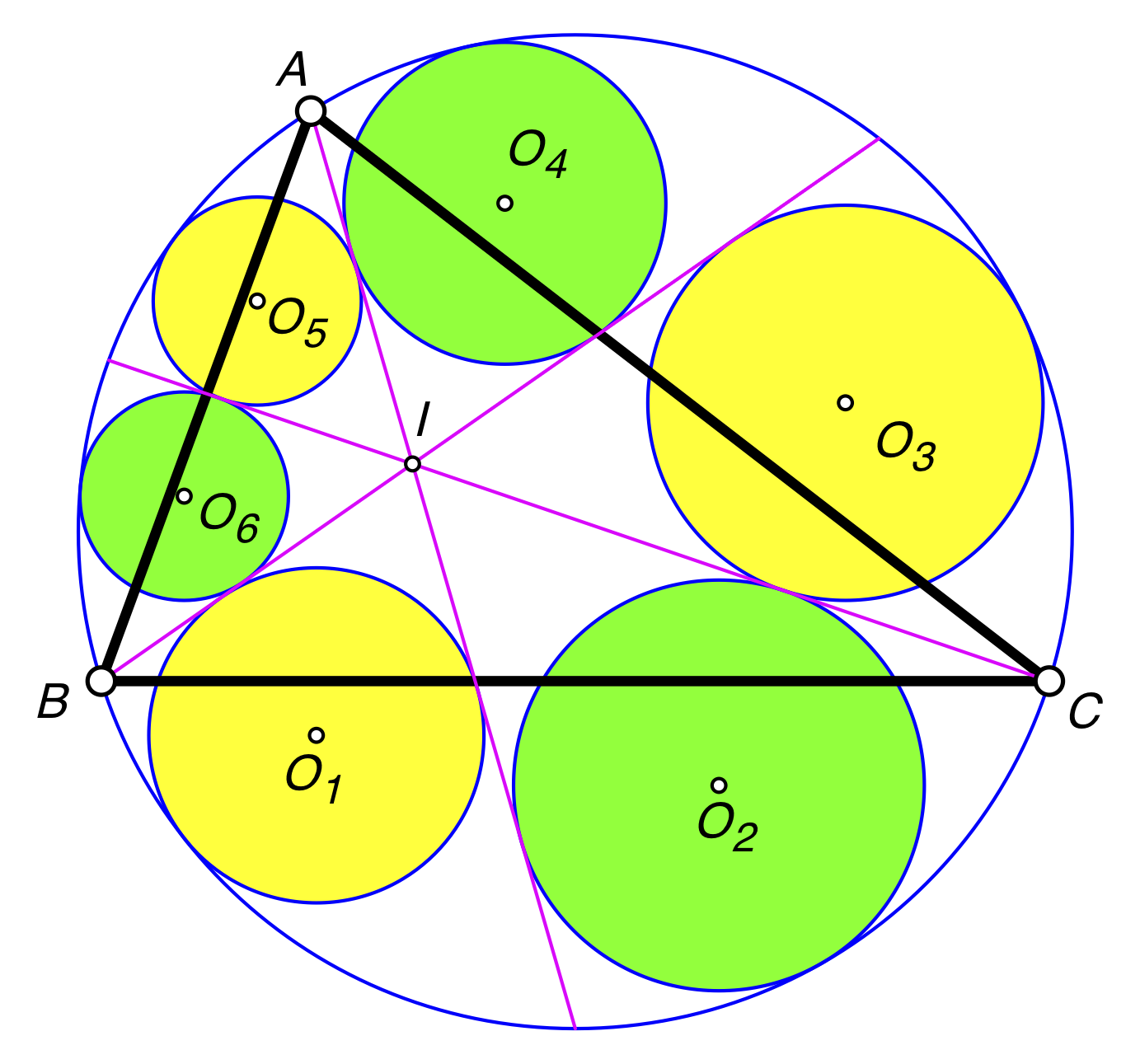}
\caption{$w_1w_3w_5=w_2w_4w_6$}
\label{fig:mixti-I}
\end{figure}

\begin{proof}
Let $\theta_i$ be the arc angle of the {\Wedge} containing circle $W_i(w_i)$.
By Theorem~\ref{thm:wwFormula}, $w_1/w_4=\tan(\theta_1/4)/\tan(\theta_4/4)$.
Similarly, $w_2/w_5=\tan(\theta_2/4)/\tan(\theta_5/4)$ and $w_3/w_6=\tan(\theta_3/4)/\tan(\theta_6/4)$.
Consequently,
$$\frac{w_1w_3w_5}{w_2w_4w_6}=\frac{w_1}{w_4}\cdot\frac{w_3}{w_6}\cdot\frac{w_5}{w_2}
=\frac{\tan(\theta_1/4)}{\tan(\theta_4/4)}\cdot\frac{\tan(\theta_3/4)}{\tan(\theta_6/4)}\cdot\frac{\tan(\theta_5/4)}{\tan(\theta_2/4)}.$$
Note that $\angle BAI=\angle CAI$ since $AI$ is an angle bisector.
Therefore, $\theta_1=\theta_2$. Similarly, $\theta_3=\theta_4$ and $\theta_5=\theta_6$.
Hence
$$\frac{w_1w_3w_5}{w_2w_4w_6}=\frac{\tan(\theta_1/4)}{\tan(\theta_3/4)}\cdot\frac{\tan(\theta_3/4)}{\tan(\theta_5/4)}\cdot\frac{\tan(\theta_5/4)}{\tan(\theta_1/4)}=1,$$
so $w_1w_3w_5=w_2w_4w_6$.
\end{proof}

\void{
The following theorem comes from \cite{Rabinowitz-in} and will be needed shortly.
\begin{theorem}
\label{thm:centroid1}
Let $M$ be the centroid of $\triangle ABC$. The medians through $M$ divide $\triangle ABC$ into six small triangles.
Circles $O_i(r_i)$ are inscribed in these six triangles as shown in Figure~\ref{fig:centroid}.
Then $$\frac{1}{r_1}+\frac{1}{r_3}+\frac{1}{r_5}=\frac{1}{r_2}+\frac{1}{r_4}+\frac{1}{r_6}.$$
\end{theorem}
\begin{figure}[h!t]
\centering
\includegraphics[width=0.5\linewidth]{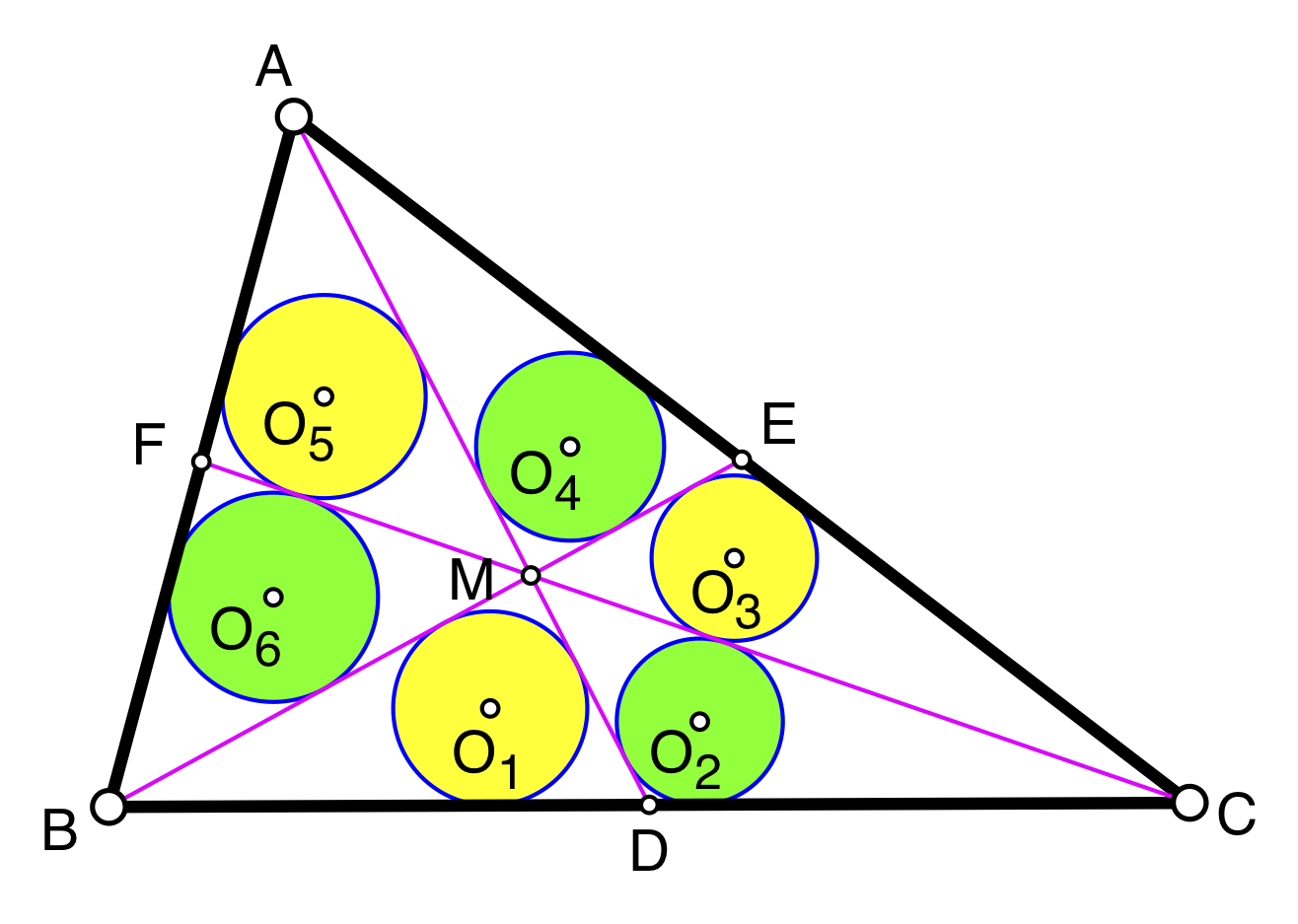}
\caption{$1/r_1+1/r_3+1/r_5=1/r_2+1/r_4+1/r_6$}
\label{fig:centroid}
\end{figure}
}

\begin{theorem}
\label{thm:segs-H}
Let $H$ be the orthocenter of acute $\triangle ABC$. The altitudes through $H$ extended to meet the circumcircle of $\triangle ABC$ divide the segments of the circumcircle bounded by the sides of the triangle into two {\Wedges} each
as shown in Figure~\ref{fig:segs-H}. 
Let $W_i(w_i)$ be the incircles of the six {\Wedges} formed,
situated as shown in Figure~\ref{fig:segs-H}.
Then $w_1w_3w_5=w_2w_4w_6$.
\end{theorem}

\begin{figure}[h!t]
\centering
\includegraphics[width=0.4\linewidth]{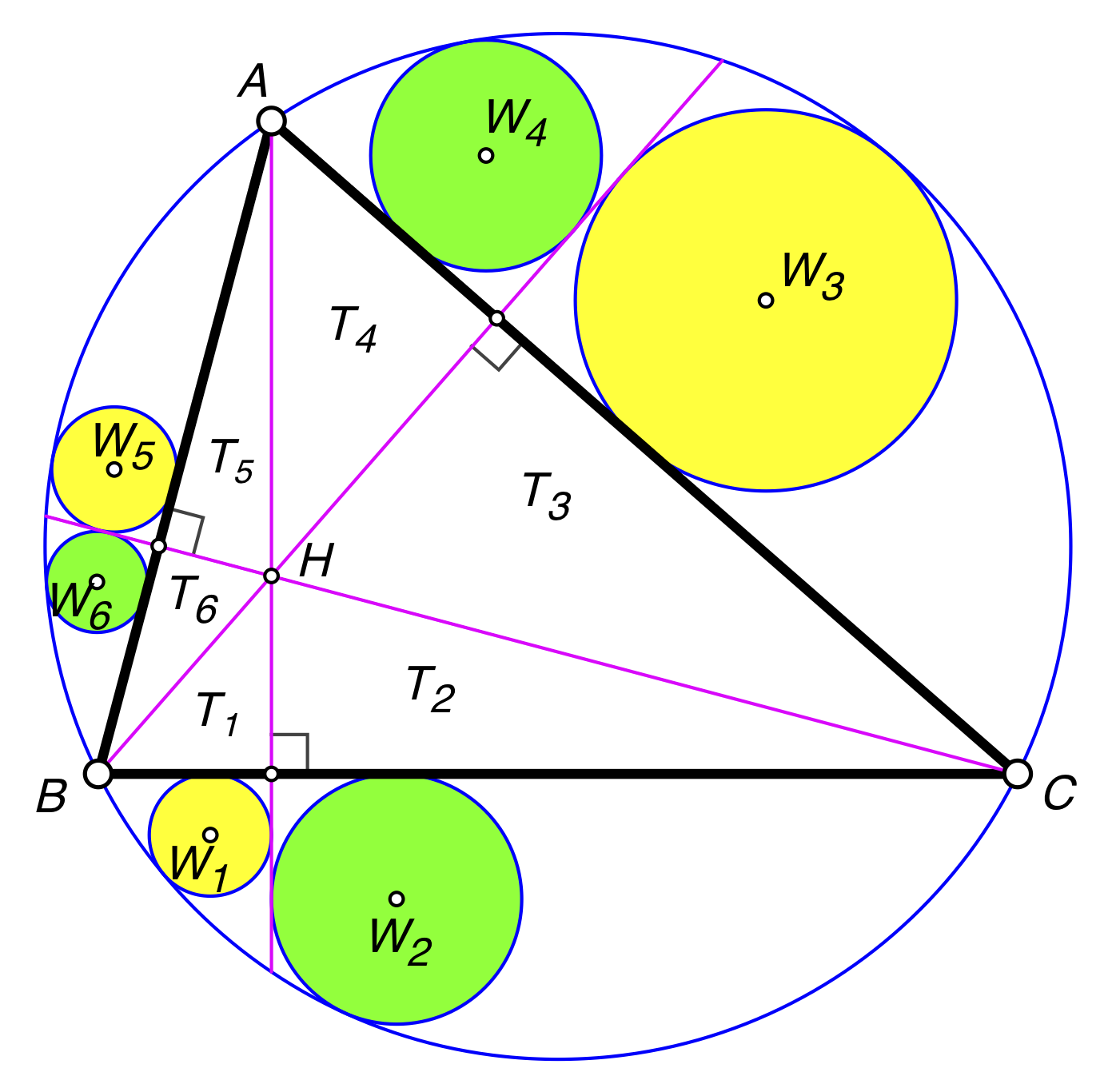}
\caption{$w_1w_3w_5=w_2w_4w_6$}
\label{fig:segs-H}
\end{figure}

\begin{proof}
The altitudes of $\triangle ABC$ divide it into six triangles named $T_1$ through $T_6$ as shown in Figure~\ref{fig:segs-H}. Let $r_i$ be the inradius of triangle $T_i$.
By Theorem~\ref{thm:segs-4Hcircles}, $w_1/r_1=w_6/r_6$. Similarly, $w_3/r_3=w_2/r_2$ and $w_5/r_5=w_4/r_4$.
Therefore,
$$\frac{w_1w_3w_5}{r_1r_3r_5}=\frac{w_6w_2w_4}{r_6r_2r_4}.$$
But $r_1r_3r_5=r_2r_4r_6$ by Theorem~3.1 of \cite{Rabinowitz-in}.
Therefore, $w_1w_3w_5=w_2w_4w_6$.
\end{proof}

\begin{theorem}
\label{thm:segs-M}
Let $M$ be the centroid of $\triangle ABC$. The medians through $M$ extended to meet the circumcircle of $\triangle ABC$ divide the segments of the circumcircle bounded by the sides of the triangle into two {\Wedges} each
as shown in Figure~\ref{fig:segs-M}. 
Let $W_i(w_i)$ be the incircles of the six {\Wedges} formed,
situated as shown in Figure~\ref{fig:segs-M}.
Then
$$\frac{1}{w_1}+\frac{1}{w_3}+\frac{1}{w_5}=\frac{1}{w_2}+\frac{1}{w_4}+\frac{1}{w_6}.$$
\end{theorem}

\begin{figure}[h!t]
\centering
\includegraphics[width=0.3\linewidth]{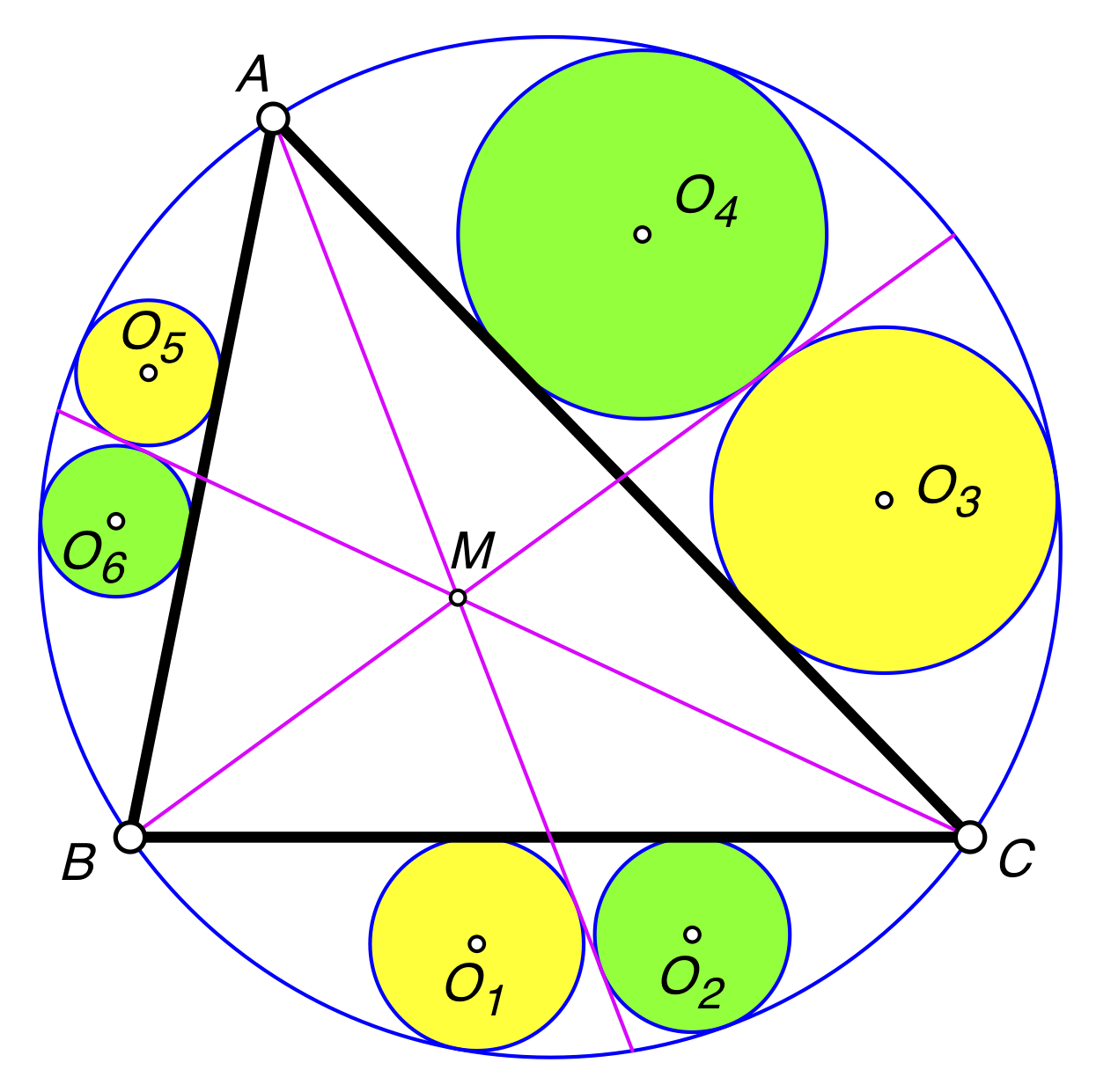}
\caption{$1/w_1+1/w_3+1/w_5=1/w_2+1/w_4+1/w_6$}
\label{fig:segs-M}
\end{figure}

\begin{proof}
A cevian through a point $P$ inside a triangle $ABC$ divides $\triangle ABC$ into two triangles, known as \textit{side triangles}.
There are six such side triangles, named $S_1$ through $S_6$ as shown in Figure~\ref{fig:sideTriangles}.

\begin{figure}[h!t]
\centering
\includegraphics[width=0.75\linewidth]{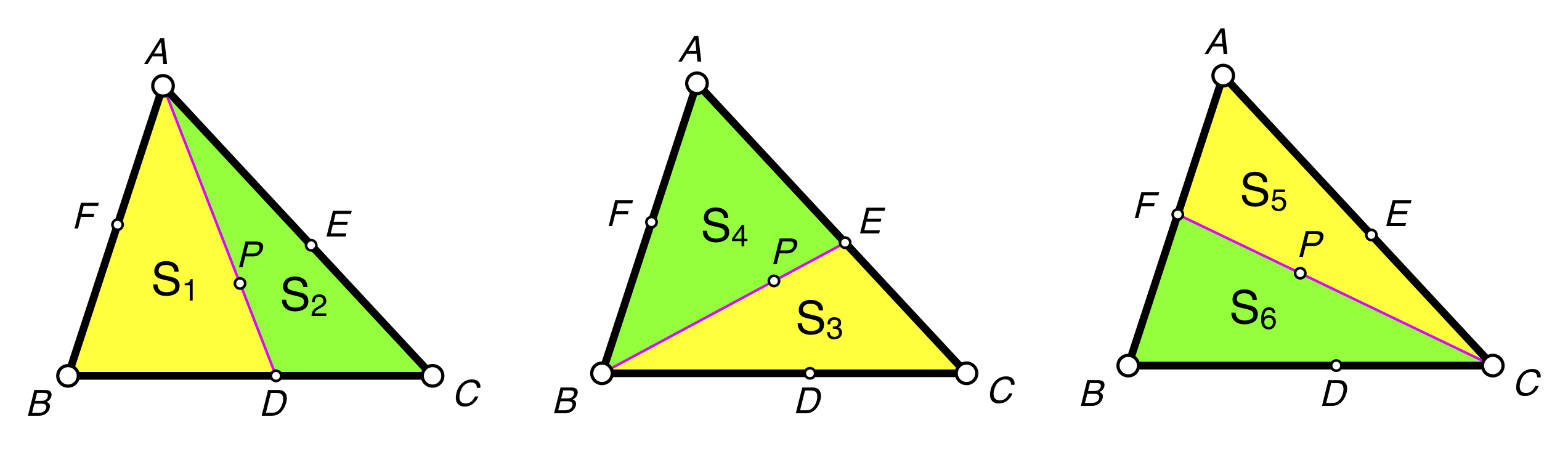}
\caption{naming of side triangles}
\label{fig:sideTriangles}
\end{figure}

Let $r_i$ be the radius of the circle inscribed in triangle $S_i$.
When $P$ is the centroid of $\triangle ABC$, Theorem~2.2 from \cite{Rabinowitz-other} states that
$$\frac{1}{r_1}+\frac{1}{r_3}+\frac{1}{r_5}=\frac{1}{r_2}+\frac{1}{r_4}+\frac{1}{r_6}.$$

By Theorem~\ref{thm:4circles},
$$
\begin{aligned}
\left(\frac{1}{w_1}-\frac{1}{w_2}\right)&+\left(\frac{1}{w_3}-\frac{1}{w_4}\right)+\left(\frac{1}{w_5}-\frac{1}{w_6}\right)\\
&=\left(\frac{1}{r_1}-\frac{1}{r_2}\right)+\left(\frac{1}{r_3}-\frac{1}{r_4}\right)+\left(\frac{1}{r_5}-\frac{1}{r_6}\right)=0,
\end{aligned}
$$
so $1/w_1+1/w_3+1/w_5=1/w_2+1/w_4+1/w_6$.
\end{proof}

\begin{theorem}
\label{thm:side-H2}
Let $H$ be the orthocenter of acute $\triangle ABC$.
The altitudes through $H$ divide the triangle into six side triangles, $S_1$ through $S_6$ as shown in
Figure~\ref{fig:sideTriangles}.
Let $W_i(w_i)$ be the incircle of the {\Wedge} associated with $S_i$.
Two of these circles are shown in Figure~\ref{fig:side-H2}.
Then $w_1w_3w_5=w_2w_4w_6$.
\end{theorem}

\begin{figure}[h!t]
\centering
\includegraphics[width=0.4\linewidth]{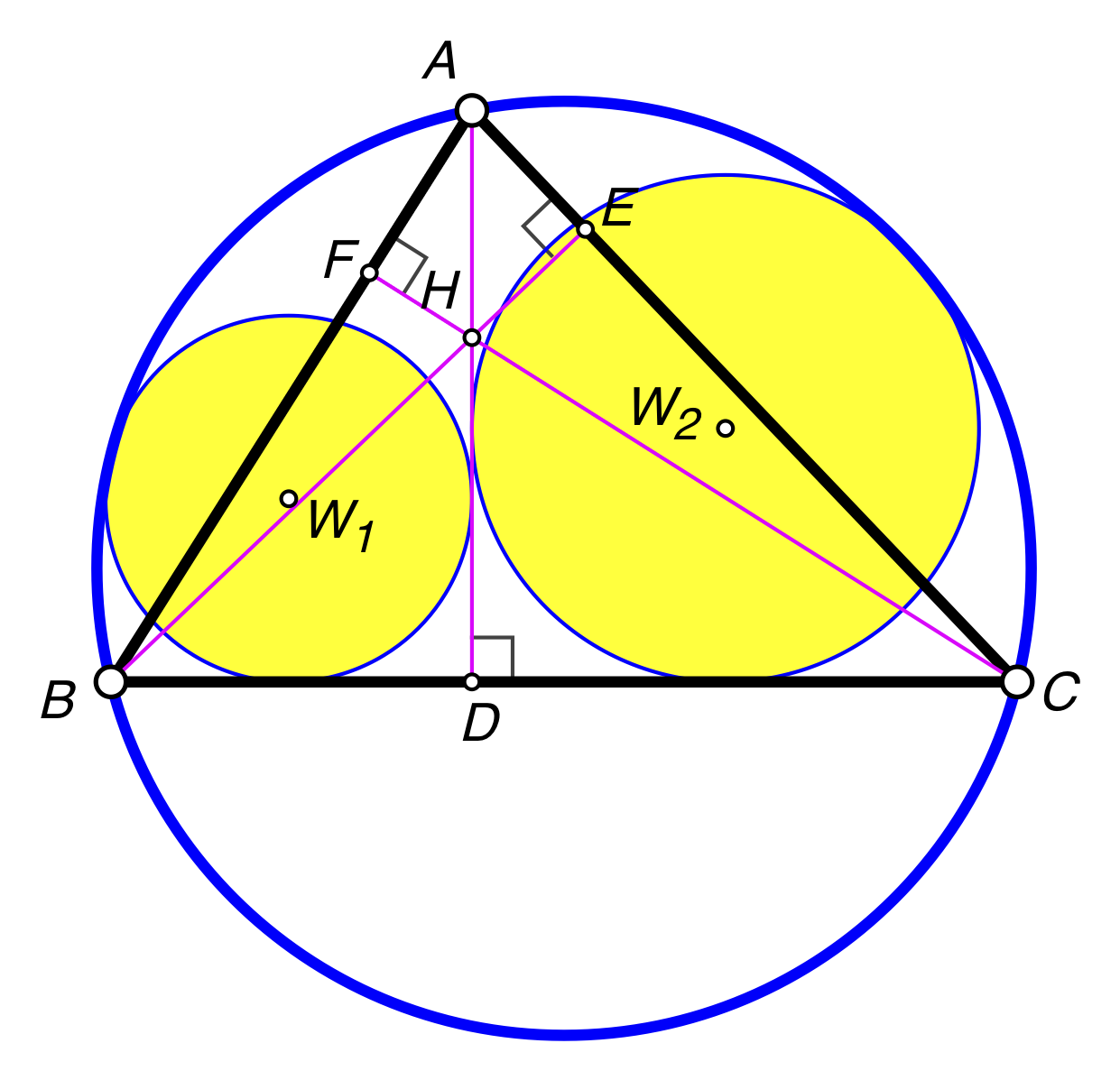}
\caption{$w_1w_3w_5=w_2w_4w_6$}
\label{fig:side-H2}
\end{figure}

\begin{proof}
This follows from Theorem~\ref{thm:segs-H} by applying Corollary~\ref{cor:segs-4H}.
\end{proof}

\begin{theorem}
\label{thm:mixti-O}
Let $O$ be the circumcenter of $\triangle ABC$. The cevians through $O$ extended to meet the circumcircle of $\triangle ABC$ divide the interior of that circumcircle into six \Wedges,
each having vertex at $O$,
as shown in Figure~\ref{fig:mixti-O}. 
Let $W_i(w_i)$ be the circle tangent to two cevians and internally tangent to
the circumcircle as shown.
Then $$w_1=w_4,\quad w_2=w_5,\quad w_3=w_6.$$
\end{theorem}

\begin{figure}[h!t]
\centering
\includegraphics[width=0.4\linewidth]{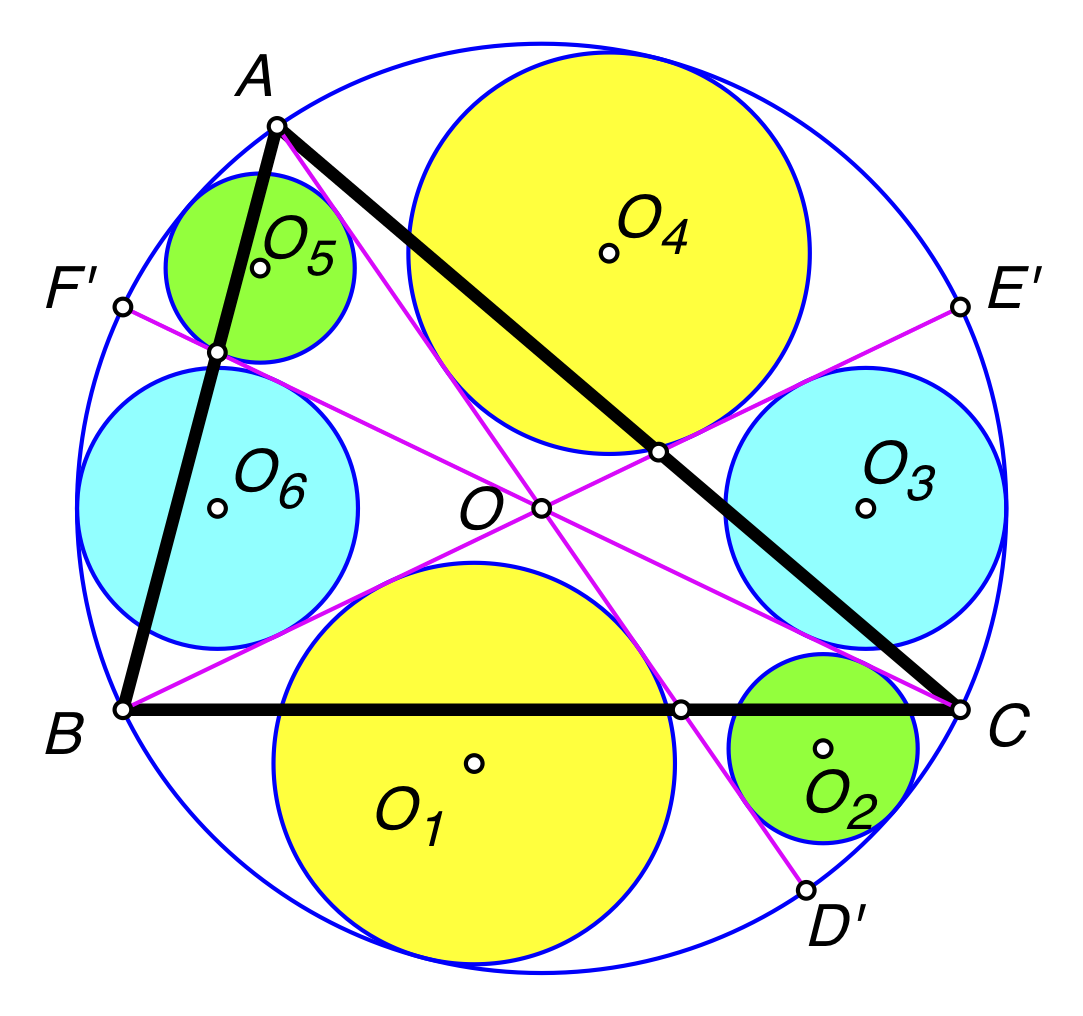}
\caption{$w_1=w_4$, $w_2=w_5$, $w_3=w_6$}
\label{fig:mixti-O}
\end{figure}

\begin{proof}
It suffices to show that $w_1=w_4$.
Note that $\angle BOD'=\angle AOE'$ because they are vertical angles.
Also, $OB=OD'=OE'=OA$ because they are all radii of circle $O$.
Therefore, {\Wedges} $BOD'$ and $AOE'$ are congruent and thus their incircles are also congruent.
\end{proof}

\begin{theorem}
\label{thm:inscribed-O}
Let $O$ be the circumcenter of $\triangle ABC$. The cevians through $O$ are extended to meet the circumcircle of $\triangle ABC$ at points $D'$, $E'$, and $F'$ as shown in Figure~\ref{fig:inscribed-O}.
The cevians divide $\triangle ABC$ into six side triangles named $S_1$ through $S_6$ as shown
in Figure~\ref{fig:sideTriangles}.
Six circles, $W_i(w_i)$, are inscribed in the {\Wedges} associated with these side triangles. Two of these circles
are shown in Figure~\ref{fig:inscribed-O}. 
Then $$w_1=w_4,\quad w_2=w_5,\quad w_3=w_6.$$
\end{theorem}

\begin{figure}[h!t]
\centering
\includegraphics[width=0.3\linewidth]{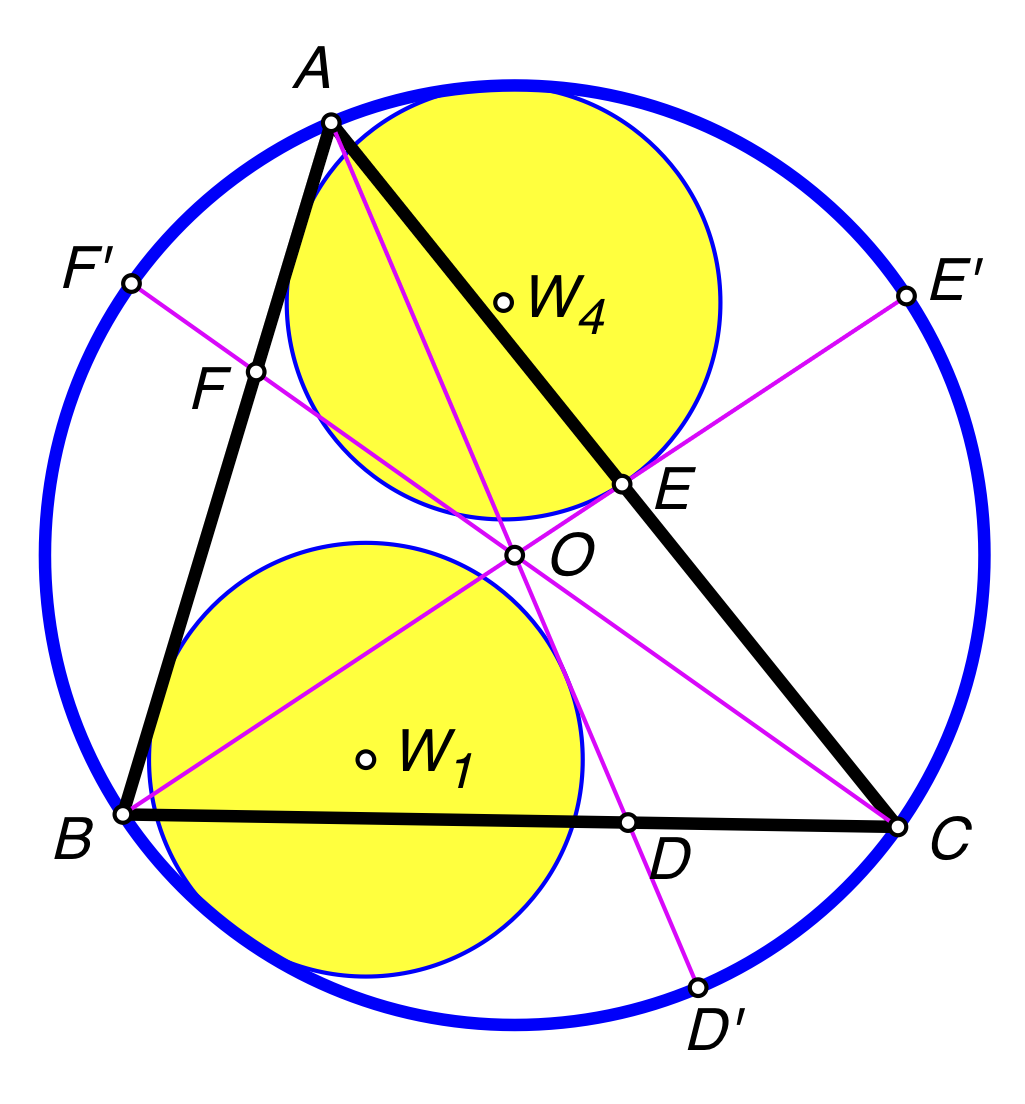}
\caption{$w_1=w_4$}
\label{fig:inscribed-O}
\end{figure}

\begin{proof}
It suffices to show that $w_1=w_4$.
Note that $OA=OB$ because they are both radii of circle $O$.
Thus, $\angle OAB=\angle OBA$ because they are the base angles of an isosceles triangle.
Also, $AD'=BE'$ because they are both diameters of circle $O$.
The {\Wedges} $BAD'$ and $ABE'$ have side $AB$ in common.
Therefore, {\Wedges} $BAD'$ and $ABE'$ are congruent and hence their incircles are also congruent.
\end{proof}


\end{document}